\documentclass[english]{article}
\usepackage[T1]{fontenc}
\usepackage[latin9]{inputenc}
\usepackage{parskip}
\usepackage{verbatim}
\usepackage{mathtools}
\usepackage{dsfont}
\usepackage{amsmath}
\usepackage{amsthm}
\usepackage{amssymb}
\usepackage{graphicx}
\usepackage{xargs}[2008/03/08]

\makeatletter
\theoremstyle{plain}
\newtheorem{thm}{\protect\theoremname}[section]
\theoremstyle{plain}
\newtheorem{conjecture}[thm]{\protect\conjecturename}
\theoremstyle{plain}
\newtheorem{cor}[thm]{\protect\corollaryname}
\theoremstyle{definition}
\newtheorem{defn}[thm]{\protect\definitionname}
\theoremstyle{plain}
\newtheorem{lem}[thm]{\protect\lemmaname}
\theoremstyle{remark}
\newtheorem{rem}[thm]{\protect\remarkname}
\theoremstyle{plain}
\newtheorem{prop}[thm]{\protect\propositionname}
\theoremstyle{plain}
\newtheorem{fact}[thm]{\protect\factname}
\theoremstyle{definition}
\newtheorem{example}[thm]{\protect\examplename}

\usepackage{hyperref}

\usepackage{cancel}
\makeatother

\usepackage{babel}
\providecommand{\conjecturename}{Conjecture}
\providecommand{\corollaryname}{Corollary}
\providecommand{\definitionname}{Definition}
\providecommand{\examplename}{Example}
\providecommand{\factname}{Fact}
\providecommand{\lemmaname}{Lemma}
\providecommand{\propositionname}{Proposition}
\providecommand{\remarkname}{Remark}
\providecommand{\theoremname}{Theorem}

\begin{document}
\title{On the critical fugacity of the hard-core model on regular bipartite
graphs}
\author{Daniel Hadas\thanks{School of Mathematical Sciences, Tel Aviv University \url{danielhadas1@mail.tau.ac.il}}\and Ron
Peled\thanks{Department of Mathematics, University of Maryland, College Park \url{peledron@umd.edu}}}
\maketitle
\begin{abstract}
We establish long-range order for the hard-core model on a finite,
regular bipartite graph above a threshold fugacity given in terms
of expansion parameters of the graph. The result applies to the $d$-dimensional
hypercube graph and, more generally, to $d$-dimensional discrete
tori of fixed side length, proving long-range order at fugacities
$\lambda\ge\Omega(\frac{\log d}{d})$. Furthermore, we use reflection
positivity to transfer the result to the lattice $\mathbb{Z}^{d}$,
verifying the long-standing belief that its critical fugacity is of
the form $d^{-1+o(1)}$ as $d\to\infty$.%
\global\long\def\inds{\mathcal{I}}%
\global\long\def\P{\mathbb{P}}%
\global\long\def\E{\mathbb{E}}%
\global\long\def\Ve{\mathcal{E}}%
\global\long\def\Vo{\mathcal{O}}%
\global\long\def\eps{\epsilon}%
\global\long\def\condon{\,|\,}%
\global\long\def\indic#1{\mathds{1}_{#1}}%
\global\long\def\hc{\mathrm{hc}}%
\global\long\def\R{\mathbb{R}}%
\global\long\def\Z{\mathbb{Z}}%
\global\long\def\N{\mathbb{N}}%
\global\long\def\ham{\mathcal{H}}%
\global\long\def\divrel{\;\|\;}%
\global\long\def\midvert{\;\middle\vert\;}%
\global\long\def\sgn{\mathrm{sgn}}%
\global\long\def\ph{\varphi}%
\global\long\def\phm{\overset{{\tiny \frown}}{\ph}}%
\global\long\def\logplam{\tilde{\lambda}}%
\global\long\def\LE{\mathrm{LE}}%
\global\long\def\Stab{\mathrm{Stab}}%
\newcommandx\refls[2][usedefault, addprefix=\global, 1=\ell, 2=L]{D_{\hspace*{-0.4ex}#2/#1}}%
\newcommandx\zedd[3][usedefault, addprefix=\global, 1=\Lambda, 2=R, 3=f]{\left\Vert #3\right\Vert _{#2|#1}}%
\global\long\def\trace{\mathrm{Tr}}%
\end{abstract}

\section{Introduction}

\subsection{The hard-core model}

Given a finite graph $G=(V,E)$ and parameter $\lambda>0$, the \emph{hard-core
model} on $G$ at \emph{fugacity} $\lambda$ is the probability measure
$\mu^{G}$ on independent sets of $G$ which assigns each independent
set probability proportional to $\lambda$ raised to the size of the
set. Formally, we identify subsets of $G$ with their characteristic
functions, writing
\begin{equation}
\Omega_{\hc}^{G}\coloneqq\{\sigma\in\{0,1\}^{V}:uv\in E\implies\sigma_{u}\sigma_{v}=0\}\label{eq:independent configurations}
\end{equation}

for the family of hard-core configurations. For $\sigma\in\{0,1\}^{V}$,
we write $|\sigma|:=\sum_{v\in V}\sigma_{v}$, and define a measure
$\zeta^{G}$ on $\Omega_{\hc}^{G}$ by
\begin{equation}
\zeta^{G}(\sigma):=\lambda^{|\sigma|}.\label{eq: zeta measure}
\end{equation}
The hard-core measure on $G$ is given by normalizing $\zeta^{G}$
to a probability measure:
\begin{equation}
\mu^{G}(\sigma)=\frac{1}{Z^{G}}\zeta^{G}(\sigma)\label{eq:hard core measure}
\end{equation}
where $Z^{G}\coloneqq\zeta^{G}(\Omega_{\hc}^{G})$ is called the \emph{partition
function.} For brevity, we have omitted the $\lambda$ dependence
from the notation. 

Do typical configurations of the hard-core model exhibit some kind
of long-range order? Conversely, do the correlations in $\sigma$
decay at long distances? At low fugacities, the sampled independent
sets tend to be sparse, with $\sigma_{u},\sigma_{v}$ being nearly
uncorrelated for vertices $u,v$ at large graph distance. The best
general result of this kind is by Weitz, who shows in \cite[Corollary 2.6]{weitz2006counting}
that the hard-core model on $G$ exhibits \emph{strong spatial mixing}
when
\begin{equation}
\lambda\le\lambda_{\Delta}:=\frac{(\Delta-1)^{\Delta-1}}{(\Delta-2)^{\Delta}},\label{eq:tree bound on fugacity}
\end{equation}

with $\Delta$ denoting the maximal degree in $G$. Here, strong spatial
mixing means that for every subset $\Lambda\subset V$, vertex $v$,
and $\tau^{1},\tau^{2}\in\Omega_{\hc}^{G}$, we have that $|\mu^{G}(\sigma_{v}=1\,|\,\sigma_{\Lambda}=\tau^{1})-\mu^{G}(\sigma_{v}=1\,|\,\sigma_{\Lambda}=\tau^{2})|$
decays to zero with the distance of $v$ from the set $\{u:\tau_{u}^{1}\neq\tau_{u}^{2}\}$,
at a rate which depends only on $\Delta$ and $\lambda$. Here and
later, we write $\sigma_{\Lambda}$ for the restriction of $\sigma$
to $\Lambda$. The bound (\ref{eq:tree bound on fugacity}) is sharp,
as it is known that such mixing fails for the hard-core model on (finite
subsets of) the $\Delta$-regular tree at any higher fugacity. We
note that
\begin{equation}
\lambda_{\Delta}\sim\frac{e}{\Delta}\quad\text{as \ensuremath{\Delta\to\infty}.}\label{eq:tree bound asymptotics}
\end{equation}

Alternative conditions implying the decay of correlations follow from
Dobrushin's uniqueness condition \cite{dobruschin1968description},
or from van den Berg--Maes' disagreement percolation \cite{bergDisagreementPercolationStudy1994},
but in terms of their dependence on the maximal degree $\Delta$,
these apply in parameter ranges smaller than (\ref{eq:tree bound on fugacity})
(see discussion in \cite[Section 3.1.1]{peled2020long}).

At high fugacities, depending on the graph structure, many forms of
long range order may emerge. These can serve as models for various
states of matter such as crystals and liquid crystals (see discussions
in \cite{he2024high}, \cite{disertori2013nematic}, \cite{hadasColumnarOrderRandom2022},
\cite{peled2022LiquidCrystalTalk}). In this paper, we focus on the
specific case that the graph $G$ is simple, connected, regular and
bipartite, with bipartition $(\Ve,\Vo)$. In this case, the largest
independent sets are $\Ve$ and $\Vo$ and one may expect that typical
independent sets sampled from $\mu^{G}$ will be mostly contained
in one of these. Our goal is to obtain new quantitative bounds on
the minimal fugacity $\lambda$ at which such long-range order arises.
The next sections detail our results: first for the lattice $\Z^{d}$,
then for discrete tori graphs and finally for regular bipartite graphs
with suitable expansion properties. We chose to begin with the lattice
result as it is simpler to state, but we point out that its proof
uses the result for finite graphs as its main ingredient.

\subsection{The hard-core model on $\protect\Z^{d}$\label{subsec:The-hard-core-model}}

In the context of statistical physics, it is natural to study the
hard-core model on the lattice $\Z^{d}$. This is the graph whose
vertex set, also denoted $\Z^{d}$, is the set of $d$-dimensional
vectors with integer coordinates, with two vectors adjacent if their
Euclidean distance is 1. On such an infinite graph, the replacement
for the probability measure (\ref{eq:hard core measure}) is the notion
of Gibbs measures. A \emph{Gibbs measure} for the hard-core model
on $\Z^{d}$ at fugacity $\lambda$ is any probability measure $\mu^{\lambda}$
on $\Omega_{\hc}^{\Z^{d}}$ with the following property: Suppose $\sigma$
is sampled from $\mu^{\lambda}$. For any finite $\Lambda\subset\Z^{d}$,
conditioned on the restriction $\sigma_{\Lambda^{c}}$, the probability
of each $\sigma$ is proportional to $\lambda^{|\sigma_{\Lambda}|}$. 

The fundamental problem is then to decide, for each fugacity $\lambda$,
whether the hard-core model admits a unique Gibbs measure or multiple
ones, with multiplicity indicating long-range order (in this way,
the infinite graph setting is more elegant than the finite graph setting,
as it has a canonical notion of long-range order). We also mention
that, using standard monotonicity (FKG) properties, there are multiple
Gibbs measures if and only if the two Gibbs measures obtained as infinite-volume
limits with even/odd boundary conditions are distinct (see \cite[Lemma 3.2]{VANDENBERG1994179}).
These latter measures are extremal and invariant to the parity-preserving
automorphisms of $\Z^{d}$.

One thus defines
\begin{align*}
\lambda_{c}^{-}(d) & :=\inf\{\lambda:\text{the hard-core model on \ensuremath{\Z^{d}} has multiple Gibbs measures at fugacity \ensuremath{\lambda}}\},\\
\lambda_{c}^{+}(d) & :=\sup\{\lambda:\text{the hard-core model on \ensuremath{\Z^{d}} has a unique Gibbs measure at fugacity \ensuremath{\lambda}}\}.
\end{align*}

Dobrushin \cite{dobruschin1968description,dobrushin1968problem} proved
that the hard-core model on $\Z^{d}$ has a \emph{phase transition}
for each $d\ge2$ in the sense that $0<\lambda_{c}^{-}(d)\le\lambda_{c}^{+}(d)<\infty$.
While it is expected that $\lambda_{c}^{-}(d)=\lambda_{c}^{+}(d)$,
this remains unknown (and there are examples of other infinite graphs
for which $\lambda_{c}^{-}<\lambda_{c}^{+}$, i.e., multiple transitions
occur \cite{brightwellNonmonotonicBehaviorHardCore1999}).

A question of enduring interest is to find the order of magnitudes
of $\lambda_{c}^{-}(d)$ and $\lambda_{c}^{+}(d)$ in the limit $d\to\infty$.
From one side, the general bound (\ref{eq:tree bound on fugacity})
applies, and shows that (also noting (\ref{eq:tree bound asymptotics}))
\begin{equation}
\lambda_{c}^{-}(d)\ge\lambda_{2d}\sim\frac{e}{2d}\quad\text{as \ensuremath{d\to\infty}.}\label{eq:Z^d critical fugacity lower bound}
\end{equation}
From the other side, while Dobrushin's result only gives an exponentially
growing bound on $\lambda_{c}^{+}(d)$, a breakthrough result of Galvin--Kahn
\cite{galvin2004phase} from 2004 showed that $\lambda_{c}^{+}(d)\to0$
as $d\to\infty$, obtaining the bound
\begin{equation}
\lambda_{c}^{+}(d)\le C\frac{\log^{3/4}d}{d^{1/4}}\label{eq:Galvin Kahn bound}
\end{equation}

for an absolute constant $C>0$. It was also said in \cite{galvin2004phase}
that natural guesses for the correct order of magnitude of $\lambda_{c}^{+}(d)$
are $\frac{\log d}{d}$ or $\frac{1}{d}$. The current state-of-the-art,
and the only improvement so far to (\ref{eq:Galvin Kahn bound}),
is by Samotij--Peled \cite{peled2014odd} who showed that
\[
\lambda_{c}^{+}(d)\le C\frac{\log^{2}d}{d^{1/3}}.
\]

Our main result on $\Z^{d}$, stated next, finally closes most of
the gap to the lower bound (\ref{eq:Z^d critical fugacity lower bound}),
showing that $\lambda_{c}^{+}(d)=d^{-1+o(1)}$ as $d\to\infty$.
\begin{thm}
\label{thm: two Gibbs measures}There exists $C>0$ such that the
hard-core model on $\Z^{d}$ admits multiple Gibbs measures at each
fugacity $\lambda>C\frac{\log d}{d}$ in dimensions $d\ge2$.
\end{thm}

As a byproduct of our analysis we also get a bound on the specific
free energy of the model, see Corollary \ref{cor:Zd_SFE} below. We
conjecture that the lower bound (\ref{eq:Z^d critical fugacity lower bound})
captures the correct asymptotics of $\lambda_{c}^{+}$:
\begin{conjecture}
\label{conj: threshold for Z^d}$\lambda_{c}^{+}(d)\sim\frac{e}{2d}$
as $d\to\infty$.
\end{conjecture}

\subsection{The hard-core model on discrete tori and other regular bipartite
graphs\label{subsec: results on finite graphs}}

On a finite simple regular bipartite graph $G=(V,E)$, with bipartition
$(\Ve,\Vo)$, we seek to show that samples from the hard-core model
at high fugacity tend to concentrate in one of the bipartition classes.
To this end, for each $\sigma\in\Omega_{\hc}^{G}$, consider the occupation
counts of the two bipartition classes $|\sigma_{\Ve}|$,$|\sigma_{\Vo}|$
and the total occupation count $|\sigma|$. We also note the trivial
lower bound, obtained by requiring that $|\sigma_{\Ve}|=0$ (or that
$|\sigma_{\Vo}|=0$),
\begin{equation}
Z^{G}\ge(1+\lambda)^{\frac{1}{2}|V|}.\label{eq:trivial partition function lower bound}
\end{equation}

\subsubsection{Discrete tori}

We now detail our long-range order results for the discrete tori graphs
$\mathbb{\Z}_{L}^{d}$, defined as follows: Let $\Z_{L}$ denote both
the set of residue classes $\Z/L\Z$ and the cycle graph of length
$L$ with vertex set $\Z/L\Z$. The Cartesian power $\mathbb{\Z}_{L}^{d}$
refers both to $(\Z/L\Z)^{d}$ and the discrete torus graph on this
vertex set (with edge set $\{\{v,v+e_{i}\}:v\in(\Z/L\Z)^{d},i\in\{1,\dots,d\}\}$
where the $e_{i}$ denote standard basis vectors. To make sure that
the graph is simple, we identify parallel edges; as a result, $\mathbb{\Z}_{L}^{d}$
is $2d$-regular when $L\ge3$ and $d$-regular when $L=2$).

The hard-core model on the \emph{hypercube graph} $\Z_{2}^{d}$ has
attracted much attention in the literature. A seminal work by Korshunov--Sapozhenko
\cite{korshunov1983binary} (see \cite{galvin2019independent} for
an expository note in English) showed that 
\begin{equation}
|\Omega_{\hc}^{\Z_{2}^{d}}|\sim2\sqrt{e}2^{\frac{1}{2}|\Z_{2}^{d}|},\label{eq: number of independent sets on the hypercube}
\end{equation}
 corresponding to the expectation that for $\lambda=1$, typical independent
sets are contained in either the even or odd partite classes, with
rare defects. A further breakthrough was achieved by Kahn \cite{kahnEntropyApproachHardCore2001}
who showed, in a suitable sense, that this kind of long-range order
arises already for $\lambda=o(1)$ as $d\to\infty$. This work was
the first to introduce the powerful entropy method to this problem,
following Kahn--Lawrenz \cite{kahn1999generalized}. This was significantly
strengthened by Galvin \cite{galvin2011threshold}, who proved long-range
order when $\lambda>\frac{C\log d}{d^{1/3}}$. The state-of-the-art
was obtained recently by Jenssen--Malekshahian--Park \cite{jenssen2026refined},
who proved long-range order for $\lambda>\frac{C\log^{2}d}{d^{1/2}}$.
In the other direction, the bound (\ref{eq:tree bound on fugacity})
applies to show the absence of long-range order at fugacities $\lambda\le\lambda_{d}\sim\frac{e}{d}$
as $d\to\infty$.

Our work establishes long-range order for the hypercube graph at fugacities
of order $\lambda>C\frac{\log d}{d}$, thus determining the critical
fugacity up to the log factor. The result applies, more generally,
to the discrete tori graphs $\mathbb{\Z}_{L}^{d}$ (yielding long-range
order for $L$ fixed as $d\to\infty$), and this will play a role
(with $L=6$) in our proof of Theorem \ref{thm: two Gibbs measures}.

We state the result as a corollary, as it is derived (in Section \ref{sec:proof_of_cor_torus})
from our main result for finite graphs, Theorem \ref{thm:main_finite}
below.
\begin{cor}
\label{cor:torus}For each $C_{0}>0$ there exist $C_{1},C_{2}>0$
such that for each dimension $d\ge2$ and even $L\ge2$: The event
\begin{equation}
B\coloneqq\left\{ \min\{|\sigma_{\Ve}|,|\sigma_{\Vo}|\}>\frac{L^{d+1}}{d^{C_{0}}}\right\} \cup\left\{ \left||\sigma|-\frac{L^{d}}{2}\frac{\lambda}{1+\lambda}\right|>\frac{L^{d+1}}{d^{C_{0}}}\right\} \label{eq:discrete tori bad event}
\end{equation}
satisfies
\begin{align}
\mu^{\mathbb{Z}_{L}^{d}}\left(B\right)\le\frac{\zeta^{\mathbb{Z}_{L}^{d}}\left(B\right)}{(1+\lambda)^{L^{d}/2}} & \le(1+\lambda)^{-\frac{L^{d}}{d^{C_{2}}}}\label{eq: discrete tori probability bound}
\end{align}
whenever $\lambda>C_{1}\frac{\log d}{d}$.%
\end{cor}

The corollary shows that independent sets sampled from the hard-core
model are ordered, in the sense of being mostly contained in a single
partite class, asymptotically almost surely as $d\to\infty$ for a
fixed $L$.

Theorem \ref{thm:main_finite} also yields the following bound, complementing
the trivial (\ref{eq:trivial partition function lower bound}):
\begin{cor}
\label{cor:Zd_SFE}There exist $C,c>0$ such that for every $\lambda>0$,
dimension $d\ge1$ and even $L\ge2$,
\begin{equation}
\frac{1}{L^{d}}\log Z^{\Z_{L}^{d}}\le\frac{1}{2}\log(1+\lambda)+\frac{C}{d}(1+\lambda)^{-cd}+\frac{C}{L^{d}}.\label{eq:torus_free_energy_per_site}
\end{equation}
\end{cor}

We remark that much more precise bounds than (\ref{eq:torus_free_energy_per_site})
were shown for the hypercube in some of the above-cited works in restricted
fugacity regimes, starting with (\ref{eq: number of independent sets on the hypercube}).
The widest fugacity regime for this type of bounds is achieved in
\cite[Theorem 5.3]{jenssen2026refined} (continuing upon \cite{jenssenIndependentSetsHypercube2020}),
applying when $\lambda>\frac{C\log^{2}d}{d^{1/2}}$ (and $\lambda=O(1)$).
For $\Z_{L}^{d}$, $L>2$ even, see \cite{jenssen2023homomorphisms}.

\subsubsection{Regular bipartite graphs}

Our main result on finite graphs depends on two expansion properties,
on the local and global scales. On the global scale, the expansion
of $G=(V,E)$ is measured by the \emph{Cheeger constant}:
\begin{equation}
h(G)\coloneqq\min\left\{ \frac{|\partial A|}{|A|}:A\subset V,0<|A|\le\frac{|V|}{2}\right\} ,\label{eq:Cheeger}
\end{equation}
where $\partial A$ denotes the edge boundary of $A$ in $G$. On
the local scale, we measure expansion in terms of the following definition
(see Section \ref{sec:Notation} for graph notations):
\begin{defn}
\label{def:loc_exp}Let $G=(V,E)$ be a finite $\delta$-regular graph.
Say that $G$ satisfies the \emph{local expansion property with parameters
$C_{\LE},M_{\LE}>0$} if there is a probability distribution over
\emph{connected} subgraphs $\boldsymbol{T}$ of $G$ with at least
one edge such that when $\boldsymbol{T}$ is sampled from this distribution:
\end{defn}

\begin{enumerate}
\item \label{enu:loc_exp_uniformity}$|E|\cdot\P(uv\in E(\boldsymbol{T}))\le M_{\LE}$
for each $uv\in E$,
\item \label{enu:loc_exp_covering}$|V|\cdot\P(N(v)\cap V(\boldsymbol{T})\neq\emptyset)\ge\delta M_{\LE}/C_{\LE}$
for each $v\in V$.
\end{enumerate}
While this definition is quite ad hoc, it is convenient to verify
for the torus graphs $\Z_{L}^{d}$, and is suited for our needs. One
may check that necessarily $C_{\LE}\ge\frac{1}{2}$ and that there
is no loss in generality in requiring $\boldsymbol{T}$ to be a tree.
Any $\delta$-regular finite graph $G$ satisfies the local expansion
property with parameters $C_{\LE}=M_{\LE}=1$ by letting $\boldsymbol{T}$
be the subgraph of $G$ consisting of a single uniformly chosen edge,
though we will often seek to satisfy the property with a larger value
for $M_{\LE}$, while keeping $C_{\LE}$ bounded.

The following lemma, which is not used in the paper, sheds additional
light on the local expansion property by showing that, at least on
a bipartite $G$, local expansion follows from ``quantitative transience''
of the simple random walk.
\begin{lem}
\label{lem: local expansion from Green function}Let $G=(V,E)$ be
a finite $\delta$-regular bipartite graph. Let $C_{0}\ge1$ and let
$M_{0}\ge2$ integer. Suppose that for every $v\in V$, a simple random
walk $\boldsymbol{W}_{0},\dots,\boldsymbol{W}_{M_{0}-1}$ started
at $\boldsymbol{W}_{0}=v$, satisfies $\E|\{i\ge1:\boldsymbol{W}_{i}=v\}|\le(C_{0}-1)/\delta$.
Then $G$ satisfies the local expansion property with parameters $C_{\LE}=C_{0}$
and $M_{\LE}=M_{0}$, by taking $\boldsymbol{T}$ to be the trace
of the walk $\boldsymbol{W}$.
\end{lem}

We may now state our main result for finite graphs. The first part
of the theorem gives an upper bound on the partition function, complementing
the trivial lower bound (\ref{eq:trivial partition function lower bound}),
while the second part indicates long range order at sufficiently high
fugacities.
\begin{thm}
\label{thm:main_finite}There exist $C,c>0$ such that the following
holds for every $\lambda,C_{\LE},M_{\LE}>0$ and integer $\delta\ge2$.
Let $G=(V,E)$ denote a finite simple $\delta$-regular bipartite
graph, with bipartition $(\Ve,\Vo)$. Suppose that $G$ satisfies
the local expansion property with parameters $C_{\LE},M_{\LE}$. Then,
first,
\begin{equation}
\frac{1}{|V|}\log Z^{G}\le\frac{1}{2}\log(1+\lambda)+\frac{CC_{\LE}}{\delta}\left((1+\lambda)^{-\frac{\delta}{CC_{\LE}}}+\frac{1}{M_{\LE}}\right).\label{eq: free energy upper bound general graphs}
\end{equation}
Second, for $r>0$, if
\begin{equation}
\log(1+\lambda)>\frac{CC_{\LE}}{\delta}\max\left\{ \log\left(\frac{C\delta}{h(G)}\cdot\frac{|V|}{r}\right),\frac{\delta}{h(G)}\cdot\frac{|V|}{r}\cdot\frac{1}{M_{\LE}}\right\} \label{eq:main_ass}
\end{equation}
then
\begin{equation}
\mu^{G}(\min\{|\sigma_{\Ve}|,|\sigma_{\Vo}|\}>r)\le\frac{\zeta^{G}(\min\{|\sigma_{\Ve}|,|\sigma_{\Vo}|\}>r)}{(1+\lambda)^{|V|/2}}\le(1+\lambda)^{-c\frac{h(G)}{\delta}r}.\label{eq:main_bound}
\end{equation}
\end{thm}

\begin{rem}
\label{rem: fixed size independent sets}The intermediate term in
(\ref{eq:main_bound}) may be harnessed to study the distribution
of independent sets of a fixed size $N<\frac{|V|}{2}$. This may be
done as follows: Let $\mu_{N}^{G}$ be the uniform distribution over
$\Omega_{\hc,N}^{G}=\{\sigma\in\Omega_{\hc}^{G}\,\colon\,|\sigma|=N\}$.
Then, for every event $B$, by noting that $|\Omega_{\hc,N}^{G}|\ge\binom{|V|/2}{N}$
(similarly to (\ref{eq:trivial partition function lower bound}))
and choosing $\lambda$ to satisfy $\frac{|V|}{2}\cdot\frac{\lambda}{1+\lambda}=N$,
\begin{equation}
\mu_{N}^{G}(B)\le\frac{|B\cap\Omega_{\hc,N}^{G}|}{\binom{|V|/2}{N}}=\frac{\zeta^{G}(B\cap\Omega_{\hc,N}^{G})}{\lambda^{N}\binom{|V|/2}{N}}=\frac{\zeta^{G}(B\cap\Omega_{\hc,N}^{G})}{(1+\lambda)^{|V|/2}\P(\text{Bin}(\frac{|V|}{2},\frac{\lambda}{1+\lambda})=N)}\le C\sqrt{|V|}\frac{\zeta^{G}(B)}{(1+\lambda)^{|V|/2}},\label{eq:fixed_size_bad_event_bound}
\end{equation}
where $\text{Bin}(n,p)$ represents a binomial random variable with
$n$ trials and success probability $p$, and we have controlled the
probability that it equals its expectation. This allows to bound $\mu_{N}^{G}(\min\{|\sigma_{\Ve}|,|\sigma_{\Vo}|\}>r)$
via Theorem \ref{thm:main_finite} (or Corollary \ref{cor:torus}),
for suitable ranges of $N$ and $r$.
\end{rem}

\begin{rem}
[Sharpness] In Subsection \ref{subsec: threshold for general graphs example}
we show that the second part of Theorem \ref{thm:main_finite} cannot
be improved without further assumptions. Precisely, we provide examples
of graphs (with expansion parameters similar to those of the hypercube)
for which Theorem \ref{thm:main_finite} captures the correct order
of magnitude of the threshold fugacity for long-range order (in a
natural sense defined there).
\end{rem}

The following lemma describes the properties of discrete tori used
(in Subsection \ref{sec:proof_of_cor_torus}) to derive Corollary
\ref{cor:torus} from Theorem \ref{thm:main_finite}.
\begin{lem}
\label{lem: expansion for torus graphs}The torus $\mathbb{\Z}_{L}^{d}$,
for all integers $d\ge1,L\ge2$, satisfies the local expansion property
with $C_{\LE}=12$ and $M_{\LE}=\frac{6L^{d}}{d}$, and has Cheeger
constant satisfying $h(\Z_{L}^{d})\ge\frac{1}{L}$.
\end{lem}

\subsection{Outline}

Section \ref{sec:Notation} establishes basic notation. Section \ref{sec:Entropy}
discusses entropy and the free energy functional which are at the
heart of our arguments for finite graphs.

In Part \ref{part:general_finite} we prove our main result for finite
graphs, Theorem \ref{thm:main_finite}: First, global expansion is
used in Section \ref{sec:phi} to reduce to the key Proposition \ref{prop:I<=00003D-Phi},
which bounds the free energy functional of a probability measure on
independent sets in terms of a kind of expected ``interface energy''
$\E\boldsymbol{\Phi}$ of a sampled configuration. An important idea
in the proof of the proposition, sparse exposure, is introduced in
Section \ref{sec:Sparse-exposure}, reducing the proof to an estimate
of ``gain'' (Section \ref{sec:The-gain-terms}) and ``loss'' (Section
\ref{sec:The-loss-term}) terms in terms of $\E\boldsymbol{\Phi}$.
The loss term is controlled via the local expansion property, which
is used to produce an efficient encoding of the information revealed
by the sparse exposure. The proof of the proposition is completed
in Section \ref{sec:Proof-of-Proposition}, where it is also commented
that the proof simplifies for special graphs, including the discrete
tori $\Z_{L}^{d}$.

Part \ref{part:torus} includes the proofs involved in specializing
our results to the case of $\Z_{L}^{d}$. This includes the proof
of the local expansion property (Section \ref{sec:Local-expansion-for})
and Cheeger constant bound (Section \ref{sec:Global-expansion-for})
for the discrete tori graphs $\Z_{L}^{d}$, as stated in Lemma \ref{lem: expansion for torus graphs}.
It also includes Section \ref{sec:proof_of_cor_torus}, where Corollary
\ref{cor:torus} on long-range order on $\Z_{L}^{d}$ is deduced from
Theorem \ref{thm:main_finite} (relying also on deviation bounds for
the binomial distribution).

Part \ref{part:Long-range-order-on} studies the hard-core model on
the lattice $\Z^{d}$, proving Theorem \ref{thm: two Gibbs measures}.
We start in Section \ref{sec:Chessboard} by recalling the chessboard
estimate and its basic properties and establishing a comparison inequality
among chessboard seminorms on tori of different side lengths (Lemma
\ref{lem:chessboard_monotonicity}). These are applied in Section
\ref{sec:long range} to prove Theorem \ref{thm: two Gibbs measures}
via a Chessboard-Peierls argument driven by our long-range order result
(Theorem \ref{thm:main_finite}) for the discrete torus graph $\Z_{L}^{d}$
with the fixed side length $L=6$ (Corollary \ref{cor:torus} also
suffices for proving Theorem \ref{thm: two Gibbs measures}, but we
use Theorem \ref{thm:main_finite} instead, as it gives a more precise
quantitative bound in (\ref{eq:mult_bound-1})). We expect that our
argument for transferring results from tori of fixed side length to
the lattice can be adapted to other reflection positive spin systems.

The paper concludes in Part \ref{part:Discussion-and-open} with further
discussion, including connections with earlier works and open questions.

Appendix \ref{sec:Local-expansion-via} presents the proof of Lemma
\ref{lem: local expansion from Green function} while Appendix \ref{sec:proofs_of_shearer}
provides two proofs of the generalized Shearer's inequality, Proposition
\ref{thm:shearer_gen}.

\section{Notation\label{sec:Notation}}

\textbf{Constants: }The symbols $c,C$ stand for sufficiently small/large
\emph{universal} positive constants, and may take different values
in different expressions, even within the same line (i.e. they serve
as an alternative to the popular $O$ and $\Omega$ notations). This
applies when $c,C$ appear without modifiers---symbols such as $c_{\alpha},C'$
are not subject to this convention.

\textbf{Natural numbers: }Denote $\N:=\{1,2,3,\ldots\}$.

\textbf{Sums and restrictions: }For $X\in[0,\infty)^{A}$ we denote
$|X|\coloneqq\sum_{v\in A}X_{v}$. When $B\subset A$, write $X_{B}$
for the restriction $(X_{v})_{v\in B}.$ We abuse notation by conflating
$X_{v}$ and $X_{\{v\}}$.

\textbf{Probability: }We use $\P$ and $\E$ to denote probabilities
and expectations respectively. The same variable will appear in a
light font when it is deterministic and in a bold font when it is
random. The variable $\sigma$, standing for a configuration, is special
in that it is sometimes suppressed in the notation, so that we write
$f$ instead of $f(\sigma)$ and $\boldsymbol{f}$ instead of $f(\boldsymbol{\sigma})$.

For an event $E$, we denote by $\indic E$ its indicator random variable.

All of our probability spaces will be finite, except in the discussion
of infinite-volume Gibbs measures.

\textbf{Graphs: }For a graph $H$, we denote its vertex set and edge
set by $V(H)$, $E(H)$ respectively. Edges are undirected. For brevity,
we denote an edge $\{u,v\}$ by $uv$.

We require all graphs in this paper to be simple, i.e., to have no
multiple edges or self loops (this is essential as in the context
of the hard-core model the regularity of the graph would be meaningless
if parallel edges are allowed). In addition, when discussing $\delta$-regular
graphs, we always assume that $\delta\ge2$.

The edge boundary of a subset $A\subset V(H)$ is denoted $\partial A=\{uv\in E(H):u\in A,v\notin A\}$.
The neighborhood of a vertex $v\in V(H)$ is denoted $N(v)=\{u\in V(H):uv\in E(H)\}$.

The symbol $G$ always denotes a graph with $V=V(G)$ and $E=E(G)$.
Whenever a bipartite graph is considered, its chosen bipartition classes
are denoted $\Ve,\Vo$.

\textbf{Tori: }The notation $\Z_{L}$ and $\Z_{L}^{d}$ were introduced
in Subsection \ref{subsec: results on finite graphs}.

\textbf{Single-site free energy}: In the context of the hard-core
model with fugacity $\lambda$, we denote
\begin{equation}
\logplam\coloneqq\log(1+\lambda).\label{eq: lambda tilde def}
\end{equation}

\section{Entropy and the free energy functional\label{sec:Entropy}}

\subsection{Shannon entropy}

Let $\boldsymbol{X}$ be a random variable. We denote the Shannon
entropy of $\boldsymbol{X}$ by
\begin{equation}
S(\boldsymbol{X})=\sum_{x}\P(\boldsymbol{X}=x)\log\frac{1}{\P(\boldsymbol{X}=x)}.
\end{equation}
In our use of conditional entropy, we distinguish conditioning on
a random variable from conditioning on an event. Precisely, given
an additional random $\boldsymbol{Y}$ and an event $E$ (of non-zero
probability---this will be tacitly assumed below when conditioning),
we define the (expected) entropy of $\boldsymbol{X}$ conditioned
on $\boldsymbol{Y}$ and $E$ as
\begin{equation}
S(\boldsymbol{X}\condon\boldsymbol{Y};E)=\sum_{x,y}\P(\boldsymbol{Y}=y,\boldsymbol{X}=x\condon E)\log\frac{1}{\P(\boldsymbol{X}=x\condon\boldsymbol{Y}=y,E)}.\label{eq: conditional entropy def}
\end{equation}
We may drop either $\boldsymbol{Y}$ or $E$ by defaulting $\boldsymbol{Y}$
to be a constant RV, and defaulting $E$ to be the whole probability
space. Note that typically $S(\boldsymbol{X}\condon E)\neq S(\boldsymbol{X}\condon\indic E)$.
The definition implies that, for random $\boldsymbol{X},\boldsymbol{Y},\boldsymbol{Z}$
and event $E$, 
\begin{equation}
S(\boldsymbol{X}\condon\boldsymbol{Y},\boldsymbol{Z};E)=\sum_{y}\P(\boldsymbol{Y}=y\condon E)\cdot S(\boldsymbol{X}\condon\boldsymbol{Z};E,\boldsymbol{Y}=y).\label{eq: law of total entropy}
\end{equation}

We shall use the basic properties of Shannon entropy as introduced
in e.g. \cite{EntropyRelativeEntropy2005,alonProbabilisticMethodEntropy}.

We use the following generalized form of the entropy chain rule:
\begin{equation}
S(\boldsymbol{X},\boldsymbol{Y}\condon\boldsymbol{Z};E)=S(\boldsymbol{X}\condon\boldsymbol{Y},\boldsymbol{Z};E)+S(\boldsymbol{Y}\condon\boldsymbol{Z};E).\label{eq: entropy chain rule}
\end{equation}
We also use the fact that conditional entropy is larger when \emph{conditioning
on less} information:
\begin{equation}
S(\boldsymbol{X}\condon\boldsymbol{Y};E)\le S(\boldsymbol{X}\condon f(\boldsymbol{Y});E)\label{eq:S_cond_on_less}
\end{equation}

We abbreviate $S(p)$ for $S(X)$ when $p=\P(X=1)=1-\P(X=0)$. We
note the following convenient inequality:
\begin{equation}
S(x)\le x\log\frac{e}{x}.\label{eq:binary_entropy_bound}
\end{equation}

A standard fact is that the entropy is subadditive $S(\boldsymbol{X},\boldsymbol{Y})\le S(\boldsymbol{X})+S(\boldsymbol{Y})$.
This is significantly extended by the following.
\begin{prop}
[generalized Shearer inequality]\label{thm:shearer_gen}%
Let $J$ be a finite index set. Let $(\boldsymbol{X}_{j})_{j\in J}$
be random variables. Suppose that $\boldsymbol{K}$ is a random subset
of $J$, independent of $\boldsymbol{X}$, with $\P(j\in\boldsymbol{K})\ge p$
for each $j\in J$. Then
\begin{equation}
S(\boldsymbol{X})\le\frac{1}{p}S((\boldsymbol{X}_{j})_{j\in\boldsymbol{K}}\condon\boldsymbol{K}).
\end{equation}
\end{prop}

This is a straightforward generalization and rephrasing of the standard
Shearer's inequality \cite{chung1986some}, see Appendix \ref{sec:proofs_of_shearer}.

\subsection{The free energy functional\label{subsec:The-free-energy}}

Let $\Omega$ be a finite set. Let $\ham:\Omega\to(-\infty,\infty]$,
not identically $\infty$ (termed the Hamiltonian). Define a measure
$\zeta_{\ham}$ on $\Omega$ by $\zeta_{\ham}(\sigma)\coloneqq e^{-\ham(\sigma)}$.
Denote $Z_{\ham}\coloneqq\zeta_{\ham}(\Omega)$ and the (Gibbs) probability
measure $\mu_{\ham}\coloneqq\frac{\zeta_{\ham}}{Z_{\ham}}$. Let $\boldsymbol{\sigma}$
be a random variable sampled from an \emph{arbitrary} probability
measure on $\Omega$. Define the (negative) free energy functional
of $\boldsymbol{\sigma}$ as
\begin{equation}
I_{\ham}(\boldsymbol{\sigma}):=S(\boldsymbol{\sigma})-\E\ham(\boldsymbol{\sigma}).\label{eq:I_H def}
\end{equation}

In analogy with conditional entropy, we define conditional versions
of the free energy functional by
\begin{align}
I_{\ham}(\boldsymbol{\sigma}\condon\boldsymbol{X};E) & \coloneqq S(\boldsymbol{\sigma}\condon\boldsymbol{X};E)-\E[\ham(\boldsymbol{\sigma})\condon E].\label{eq:I_H_cond def}
\end{align}
The following fact, which may be verified directly, implies that $I_{\ham}$
is maximized when $\boldsymbol{\sigma}$ is sampled from $\mu_{\ham}$,
see e.g. \cite[Chapter 15]{Georgii2011} or \cite[Lemma 6.74]{friedli_velenik_2017}.
\begin{fact}
[Variational principle]\label{thm:Variational-principle}Let $\boldsymbol{\sigma}$
be sampled from an \emph{arbitrary} probability measure $\mu$ on
$\Omega$. Then
\begin{equation}
I_{\ham}(\boldsymbol{\sigma})=\log Z_{\ham}-d_{\mathrm{KL}}(\mu\divrel\mu_{\ham}).\label{eq:var_princip}
\end{equation}
where $d_{\mathrm{KL}}(\mu\divrel\mu_{\ham})=\sum_{\sigma\in\Omega}\mu(\sigma)\log\frac{\mu(\sigma)}{\mu_{\ham}(\sigma)}\ge0$
stands for the Kullback--Leibler divergence.
\end{fact}

In the sequel, we use the following consequence%
: when $\boldsymbol{\sigma}$ is sampled from $\mu_{\ham}$, for each
event $E\subset\Omega$,
\begin{equation}
\log\zeta_{\ham}(E)=I_{\ham}(\boldsymbol{\sigma}\condon E).\label{eq:I-zeta}
\end{equation}

This follows from Fact \ref{thm:Variational-principle} by noting
that $I_{\ham}(\boldsymbol{\sigma}\condon E)=I_{\ham|_{E}}(\boldsymbol{\sigma}\condon E)=\log Z_{\ham|_{E}}=\log\zeta_{\ham}(E)$,
as $\mu_{\ham}$ conditioned on $E$ is the Gibbs measure of the restriction
of the Hamiltonian $\ham$ to $E$.

\subsection{The free energy functional of the hard-core model}

We now specialize the free energy functional to the hard-core model
on a graph $G$. Let $\lambda>0$. We set $\Omega=\Omega_{\hc}^{G}$
and $\ham(\sigma)=-|\sigma|\log\lambda$, so that the measure $\mu_{\ham}$
is the hard-core measure $\mu^{G}$ defined in (\ref{eq:hard core measure}).
For brevity, in the rest of the paper, we will simply write $I$ for
this $I_{\ham}$, omitting mention of the graph $G$ and the fugacity
$\lambda$. We thus have that for any random $\boldsymbol{\sigma}\in\Omega_{\hc}^{G}$,
\begin{equation}
I(\boldsymbol{\sigma})=S(\boldsymbol{\sigma})+(\log\lambda)\E|\boldsymbol{\sigma}|.\label{eq:spec_ham}
\end{equation}

Given a subset $A\subset V$, the restriction $\boldsymbol{\sigma}_{A}$
is a random independent set in the induced subgraph $G[A]$ (i.e.,
$\boldsymbol{\sigma}_{A}\in\Omega_{\hc}^{G[A]}$), giving meaning
to $I(\boldsymbol{\sigma}_{A})$. Moreover, since the term $\E|\boldsymbol{\sigma}|$
in (\ref{eq:spec_ham}) is additive in the domain of $\boldsymbol{\sigma}$,
it follows that $I$ inherits the subadditivity of $S$, i.e.,
\begin{equation}
I(\boldsymbol{\sigma}_{A\cup B})\le I(\boldsymbol{\sigma}_{A})+I(\boldsymbol{\sigma}_{B})\quad\text{for disjoint \ensuremath{A,B\subset V}}.\label{eq: subadditivity of I}
\end{equation}
In addition, the generalized Shearer inequality of Proposition \ref{thm:shearer_gen}
applies with $I$ replacing $S$, with $(\boldsymbol{\sigma}_{v})_{v\in V}$
taking the place of $(\boldsymbol{X}_{j})_{j\in J}$ and with the
caveat that the assumption $\P(j\in\boldsymbol{K})\ge p$ of that
proposition is replaced by the \emph{equality} $\P(v\in\boldsymbol{K})=p$
for all $v\in V$ (as $\log\lambda$ may be negative).

For $v\in V$, we have the basic bound $I(\boldsymbol{\sigma}_{v})\le\tilde{\lambda}$
which may be viewed as a special case of the variational principle,
Fact \ref{thm:Variational-principle}, with $\tilde{\lambda}=\log(1+\lambda)$
as in (\ref{eq: lambda tilde def}). The following is a stronger,
conditional version.
\begin{prop}
\label{prop: I<=00003Dlogplam}If $\sigma_{v}\le\boldsymbol{X}$ for
a random $\boldsymbol{X}\in\{0,1\}$ then, for random $\boldsymbol{Y}$
and event $E$,
\begin{equation}
I(\boldsymbol{\sigma}_{v}\condon\boldsymbol{X},\boldsymbol{Y};E)\le\logplam\E[\boldsymbol{X}\condon E].
\end{equation}
\end{prop}

\begin{proof}
To see this, note first that $I(\boldsymbol{\sigma}_{v}\condon\boldsymbol{X},\boldsymbol{Y};E)=\P(\boldsymbol{X}=0\condon E)I(\boldsymbol{\sigma}_{v}\condon\boldsymbol{Y};E,\boldsymbol{X}=0)+\P(\boldsymbol{X}=1\condon E)I(\boldsymbol{\sigma}_{v}\condon\boldsymbol{Y};E,\boldsymbol{X}=1)$
by (\ref{eq: law of total entropy}) and then use that $I(\boldsymbol{\sigma}_{v}\condon\boldsymbol{Y};E,\boldsymbol{X}=0)=0$
by (\ref{eq:spec_ham}) and the fact that $\boldsymbol{\sigma}_{v}=0$
when $\boldsymbol{X}=0$ and that $I(\boldsymbol{\sigma}_{v}\condon\boldsymbol{Y};E,\boldsymbol{X}=1)\le\tilde{\lambda}$.%
\end{proof}

\part{Long-range order for regular bipartite finite graphs\label{part:general_finite}}

\section{\label{sec:phi}Coarse graining and global expansion}

In this section we deduce Theorem \ref{thm:main_finite} from Proposition
\ref{prop:I<=00003D-Phi} below, whose proof is the subject of the
next sections.

Let $G$ be a $\delta$-regular bipartite finite graph. Recall our
convention that $G=(V,E)$ with bipartition classes $\Ve,\Vo\subset V$.
When applicable, we will use the convention that $v$ denotes and
element of $\Ve$ while $u$ denotes an element of $\Vo$. Given a
configuration $\sigma\in\Omega_{\hc}^{G}$ we define a function $\ph$
which we like to think of as a ``coarse grained'' or ``smoothed''
version of $\sigma$. On the even side, define:
\begin{equation}
\ph_{v}\coloneqq\prod_{u\in N(v)}(1-\sigma_{u})\quad\quad\forall v\in\Ve,\label{eq:phi}
\end{equation}
so that $\ph_{v}$ is the indicator function of the event that all
neighbors of $v$ are vacant in $\sigma$ (so that $v$ may be occupied
in $\sigma$). We also define two useful extensions of $\ph$ to the
odd side
\begin{align}
\ph_{u} & \coloneqq\mathrm{Majority}(\ph_{v}:v\in N(u))\quad\quad & \forall u\in\Vo,\label{eq:phi_o}\\
\phm_{u} & \coloneqq\frac{1}{\delta}\sum_{v\in N(u)}\ph_{v}\quad\quad & \forall u\in\Vo,\label{eq:phi_o2}
\end{align}
breaking ties arbitrarily. Also, still given $\sigma\in\Omega_{\hc}^{G}$,
define
\begin{equation}
\Phi\coloneqq\frac{1}{|E|}\sum_{uv\in E}|\ph_{u}-\ph_{v}|=\frac{2}{|V|}\sum_{u\in\Vo}\min\{\phm_{u},1-\phm_{u}\}\label{eq:Phi}
\end{equation}
as a measure of ``roughness'' for $\ph$ (and by proxy, also for
$\sigma$) or a measure of ``the (normalized) total \mbox{(hyper-)surface}
area of domain walls between the even-occupied phase and the odd-occupied
phase''.

\subsection{Using the global expansion}

Within Section \ref{sec:phi}, denote
\begin{equation}
\mathcal{M}\coloneqq\min\{|\sigma_{\Ve}|,|\sigma_{\Vo}|\},\label{eq:M}
\end{equation}

and recall the Cheeger constant $h(G)$ from (\ref{eq:Cheeger}).
To prove Theorem \ref{thm:main_finite}, we need to bound the event
$\{\mathcal{M}>r\}$. The following lemma will allow us to work in
terms of $\Phi$ instead of $\mathcal{M}$.
\begin{lem}
\label{lem:M<=00003DPhi}Let $G$ be a finite $\delta$-regular bipartite
graph. Let $\sigma\in\Omega_{\hc}^{G}$. Then
\begin{equation}
\mathcal{M}\le\frac{\delta}{2h(G)}\Phi|V|.
\end{equation}
\end{lem}

\begin{proof}
Introduce
\begin{equation}
\mathcal{M}'\coloneqq\min\{\sum_{v\in\Ve}\ph_{v},\sum_{v\in\Ve}(1-\ph_{v})\}.
\end{equation}
Note that $|\sigma_{\Ve}|=\sum_{v\in\Ve}\sigma_{v}\le\sum_{v\in\Ve}\ph_{v}$
since by the hard-core restriction, for each $v\in\Ve$, $\sigma_{v}=1$
implies $\ph_{v}=1$. Also, $|\sigma_{\Vo}|=\sum_{u\in\Vo}\sigma_{u}\le\sum_{v\in\Ve}(1-\ph_{v})$
by a double counting argument: 
\begin{equation}
\delta\sum_{u\in\Vo}\sigma_{u}=\sum_{u\in\Vo}\sum_{v\in N(u)}\sigma_{u}\le\sum_{u\in\Vo}\sum_{v\in N(u)}(1-\ph_{v})=\delta\sum_{v\in\Ve}(1-\ph_{v})
\end{equation}
since for adjacent $u\in\Vo,v\in\Ve$, if $\sigma_{u}=1$ then $\ph_{v}=0$.
Thus $\mathcal{M}\le\mathcal{M}'$.

Now let $A$ be the smaller of the two sets $\{v\in V:\ph_{v}=0\}$
and $\{v\in V:\ph_{v}=1\}$. Note $|A|\ge\mathcal{M}'$. By the definition
of the Cheeger constant (\ref{eq:Cheeger}), $h(G)\le\frac{|\partial A|}{|A|}=\frac{|E|\Phi}{|A|}\le\frac{|E|\Phi}{\mathcal{M}'}$.
Thus $\mathcal{M}\le\mathcal{M}'\le\frac{\delta}{2h(G)}\Phi|V|$.
\end{proof}

\subsection{Bounding the free energy functional by $\protect\E\boldsymbol{\Phi}$
and deduction of the main result}

Under sufficiently strong local expansion as quantified by Definition
\ref{def:loc_exp}, we show that if $\boldsymbol{\sigma}$ is distributed
such that $\E\boldsymbol{\Phi}$ is sufficiently large, then $\E\boldsymbol{\Phi}$
contributes a ``surface tension'' term to the free energy functional
$I(\boldsymbol{\sigma})$. Recall from (\ref{eq: lambda tilde def})
that $\logplam=\log(1+\lambda)$.
\begin{prop}
\label{prop:I<=00003D-Phi}There exist $C,c>0$ such that the following
holds. Let $G$ be a finite $\delta$-regular bipartite graph, satisfying
the local expansion property with parameters \emph{$C_{\LE},M_{\LE}$}.
Let $\boldsymbol{\sigma}$ be sampled from \emph{any measure} on $\Omega_{\hc}^{G}$.
Let $\lambda>0$. Then
\begin{equation}
\frac{1}{|V|}I(\boldsymbol{\sigma})-\frac{1}{2}\logplam\le-\frac{1}{4}\logplam\E\boldsymbol{\Phi}+\frac{CC_{\LE}}{\delta}\left(\E\boldsymbol{\Phi}\log\frac{e}{\E\boldsymbol{\Phi}}+\frac{1}{M_{\LE}}\right).\label{eq:I<=00003DPhi bound}
\end{equation}
Consequently,
\begin{equation}
I(\boldsymbol{\sigma})-\logplam\frac{|V|}{2}\le-c\logplam|V|\E\boldsymbol{\Phi}
\end{equation}
 whenever 
\begin{equation}
\logplam>\frac{9CC_{\LE}}{\delta}\max\left\{ \log\frac{e}{\E\boldsymbol{\Phi}},\frac{1}{M_{\LE}\E\boldsymbol{\Phi}}\right\} .\label{eq:I<=00003D-Phi assumption}
\end{equation}
\end{prop}

This proposition suffices for the
\begin{proof}
[Proof of Theorem \ref{thm:main_finite}]The free energy bound (\ref{eq: free energy upper bound general graphs})
is obtained as follows: Let $\boldsymbol{\sigma}$ be sampled from
$\mu^{G}$. Note that $\log Z^{G}=I(\boldsymbol{\sigma})$ by the
variational principle, Fact \ref{thm:Variational-principle}. The
bound follows from Proposition \ref{prop:I<=00003D-Phi} by maximizing
the right-hand side of (\ref{eq:I<=00003DPhi bound}) over $\E\boldsymbol{\Phi}\in[0,1]$.

The first inequality of (\ref{eq:main_bound}) follows from (\ref{eq:trivial partition function lower bound}).
For the second inequality, let $\boldsymbol{\sigma}$ be sampled from
$\mu^{G}$ conditioned on the event $\boldsymbol{\mathcal{M}}>r$.
By (\ref{eq:I-zeta}), it holds that $\log\zeta^{G}(\boldsymbol{\mathcal{M}}>r)=I(\boldsymbol{\sigma})$.
Since $\E\boldsymbol{\mathcal{M}}>r$, Lemma \ref{lem:M<=00003DPhi}
implies that $\frac{|V|}{h(G)r}>\frac{2}{\delta\E\boldsymbol{\Phi}}$.
Together with the theorem's assumption (\ref{eq:main_ass}), for sufficiently
large $C$, this implies the assumption (\ref{eq:I<=00003D-Phi assumption})
of Proposition \ref{prop:I<=00003D-Phi}. We get the second inequality
as follows:
\begin{align}
\log\zeta^{G}(\boldsymbol{\mathcal{M}}>r) & =I(\boldsymbol{\sigma})\le\logplam\left(\frac{|V|}{2}-c|V|\E\boldsymbol{\Phi}\right)\nonumber \\
 & \le\logplam\left(\frac{|V|}{2}-c\frac{h(G)}{\delta}r\right)=\log\left[(1+\lambda)^{|V|/2-c\frac{h(G)}{\delta}r}\right].\qedhere
\end{align}
\end{proof}
The remainder of Part \ref{part:general_finite} is about proving
Proposition \ref{prop:I<=00003D-Phi}. The proof may be simplified
for a certain family of graphs $G$, which includes the tori $\Z_{L}^{d}$.
The simplifications are described in the final Subsection \ref{subsec:Simplified-proof},
and it may be beneficial to keep them in mind, especially on a first
read.

\section{Sparse exposure\label{sec:Sparse-exposure}}

Our strategy for upper bounding $I(\boldsymbol{\sigma})$ in Proposition
\ref{prop:I<=00003D-Phi} starts with the next lemma which provides
a useful decomposition as a sum of three terms, which we will bound
separately. This decomposition holds in general, without assumptions
on expansion.

The random set $\boldsymbol{A}$ appearing in it may be thought of
as a set on which we ``reveal'' part of the information, specifically
$\boldsymbol{\ph}_{\boldsymbol{A}}$, providing a starting point for
our later bounds. It may be useful to keep in mind that in our intended
application the density of $\boldsymbol{B}$ in $\Vo$ will be at
least a universal constant, and the density of $\boldsymbol{A}$ in
$\Ve$ will be of order $1/\delta$.
\begin{lem}
\label{lem:three term decomposition}Let $G$ be a finite $\delta$-regular
bipartite graph. Let $\boldsymbol{\sigma}$ be sampled from \emph{any
measure} on $\Omega_{\hc}^{G}$. Let $\boldsymbol{A}\subset\Ve$ and
$\boldsymbol{B}\subset\Vo$ be random subsets, such that $(\boldsymbol{A},\boldsymbol{B})$
is independent of $\boldsymbol{\sigma}$ and $\P(u\in\boldsymbol{B})$
is the same for all $u\in\Vo$ and non-zero. Denote by $s:=\P(u\in\boldsymbol{B})$
its common value. Then
\begin{equation}
I(\boldsymbol{\sigma})\le I(\boldsymbol{\sigma}_{\Ve}\condon\boldsymbol{\sigma}_{\Vo})+\frac{1}{s}I(\boldsymbol{\sigma}_{\boldsymbol{B}}\condon\boldsymbol{B},\boldsymbol{A},\boldsymbol{\ph}_{\boldsymbol{A}})+\frac{1}{s}S(\boldsymbol{\ph}_{\boldsymbol{A}}\condon\boldsymbol{A}).\label{eq:three term decomposition}
\end{equation}
\end{lem}

\begin{proof}
It suffices to justify the following steps:
\begin{align}
I(\boldsymbol{\sigma}) & =I(\boldsymbol{\sigma}_{\Ve}\condon\boldsymbol{\sigma}_{\Vo})+I(\boldsymbol{\sigma}_{\Vo})\nonumber \\
 & \le I(\boldsymbol{\sigma}_{\Ve}\condon\boldsymbol{\sigma}_{\Vo})+\frac{1}{s}I(\boldsymbol{\sigma}_{\boldsymbol{B}}\condon\boldsymbol{B})\nonumber \\
 & =I(\boldsymbol{\sigma}_{\Ve}\condon\boldsymbol{\sigma}_{\Vo})+\frac{1}{s}\left[I(\boldsymbol{\sigma}_{\boldsymbol{B}}\condon\boldsymbol{B},\boldsymbol{A},\boldsymbol{\ph_{A}})+S(\boldsymbol{\ph_{A}}\condon\boldsymbol{A})\right].\label{eq:shearer-1-1}
\end{align}
First, expand the definition of $I$ (\ref{eq:spec_ham}) in all of
its occurrences. For the $(\log\lambda)\E|\sigma|$ terms, we have
equalities throughout, using the assumption that $s=\P(u\in\boldsymbol{B})$
for all $u$. It remains to prove the display above with all occurrences
of $I$ replaced by $S$. The inequality follows from the generalized
Shearer's inequality (using that $s\le\P(u\in\boldsymbol{B})$ for
all $u$). The two equalities follow from the entropy chain rule (\ref{eq: entropy chain rule}),
where for the second equality we also use that $S(\boldsymbol{\sigma}_{\boldsymbol{B}}\condon\boldsymbol{B})=S(\boldsymbol{\sigma}_{\boldsymbol{B}}\condon\boldsymbol{B},\boldsymbol{A})$
and $S(\boldsymbol{\ph_{A}}\condon\boldsymbol{A})=S(\boldsymbol{\ph_{A}}\condon\boldsymbol{B},\boldsymbol{A})$,
both of which follow, using the definition of entropy, from the fact
that $(\boldsymbol{A},\boldsymbol{B})$ is independent of $(\boldsymbol{\sigma},\boldsymbol{\ph})$
. 
\end{proof}
Among the three terms appearing in the right-hand side of (\ref{eq:three term decomposition}),
we regard the first two as ``gain terms'' whose bound will improve
when $\E\boldsymbol{\Phi}$ is large, and the third term as a ``loss''
or error term which we will seek to control sufficiently tightly.
The next two sections are devoted to bounding these terms.

We point out the special case that $\boldsymbol{B}=\Vo$ deterministically
(and thus $s=1$). In this case we get that (\ref{eq:three term decomposition})
is an \emph{equality}. In addition, the intuition that we are revealing
$\boldsymbol{\ph}_{\boldsymbol{A}}$ is more obvious, and there is
no need to apply Shearer's inequality. An application of this special
case is described in Subsection \ref{subsec:Simplified-proof}.

\section{The gain terms\label{sec:The-gain-terms}}

The next lemma bounds the first two terms on the right-hand side of
(\ref{eq:three term decomposition}) (the gain terms). First, it is
shown in (\ref{eq:local gain terms}) that they are bounded by a sum
of local expressions. Second, we bound these local terms by relying
on the following property of the distribution of $(\boldsymbol{A},\boldsymbol{B})$.
We say that $(\boldsymbol{A},\boldsymbol{B})$ satisfies the \emph{two
uniform neighbors property} if the following holds for every $u\in\Vo$:
The distribution of $\boldsymbol{A}$, conditioned on $u\in\boldsymbol{B}$,
can be coupled with $(\boldsymbol{v}_{u,1},\boldsymbol{v}_{u,2})$,
a pair of independently uniformly sampled neighbors of $u$, such
that $\{\boldsymbol{v}_{u,1},\boldsymbol{v}_{u,2}\}\subset\boldsymbol{A}$.
\begin{lem}
\label{lem: gain terms estimate}Under the assumptions of Lemma \ref{lem:three term decomposition}:
\begin{align}
 & I(\boldsymbol{\sigma}_{\Ve}\condon\boldsymbol{\sigma}_{\Vo})+\frac{1}{s}I(\boldsymbol{\sigma}_{\boldsymbol{B}}\condon\boldsymbol{B},\boldsymbol{A},\boldsymbol{\ph}_{\boldsymbol{A}})-\logplam\frac{|V|}{2}\label{eq:local gain terms}\\
 & \overset{(1)}{\le}\logplam\sum_{u\in\Vo}\E\left[\prod_{v\in N(u)\cap\boldsymbol{A}}(1-\boldsymbol{\ph}_{v})+\boldsymbol{\phm}_{u}-1\condon u\in\boldsymbol{B}\right]\\
 & \overset{(2)}{\le}-\logplam\sum_{u\in\Vo}\E[\boldsymbol{\phm}_{u}(1-\boldsymbol{\phm}_{u})]\overset{(3)}{\le}-\frac{1}{4}\logplam|V|\E\boldsymbol{\Phi},
\end{align}

where we assume that $(\boldsymbol{A},\boldsymbol{B})$ satisfies
the two uniform neighbors property for inequality (2).
\end{lem}

\begin{proof}
We first prove inequality (1). For the first term on the LHS, by the
subadditivity (\ref{eq: subadditivity of I}) of $I$, Proposition
\ref{prop: I<=00003Dlogplam}, and the $\delta$-regularity of $G$,
we have
\begin{equation}
I(\boldsymbol{\sigma}_{\Ve}\condon\boldsymbol{\sigma}_{\Vo})\le\sum_{v\in\Ve}I(\boldsymbol{\sigma}_{v}\condon\boldsymbol{\ph}_{v})\le\sum_{v\in\Ve}\logplam\E\boldsymbol{\ph}_{v}=\logplam\sum_{u\in\Vo}\E\boldsymbol{\phm}_{u}.\label{eq:abst_even-1-1}
\end{equation}
For the second term,
\begin{align*}
I(\boldsymbol{\sigma}_{\boldsymbol{B}}\condon\boldsymbol{B},\boldsymbol{A},\boldsymbol{\ph}_{\boldsymbol{A}}) & \le\sum_{u\in\Vo}I(\boldsymbol{\sigma}_{\{u\}\cap\boldsymbol{B}}\condon\boldsymbol{B},\boldsymbol{A},\boldsymbol{\ph}_{\boldsymbol{A}})\\
(\text{By (\ref{eq: law of total entropy}) and \ensuremath{I(\boldsymbol{\sigma}_{\{u\}\cap\boldsymbol{B}}\condon\boldsymbol{B},\boldsymbol{A},\boldsymbol{\ph}_{\boldsymbol{A}};u\notin\boldsymbol{B})=0}}) & =\sum_{u\in\Vo}\P(u\in\boldsymbol{B})I(\sigma_{u}\condon\boldsymbol{B},\boldsymbol{A},\boldsymbol{\ph}_{\boldsymbol{A}};u\in\boldsymbol{B})\\
(\text{as \ensuremath{s=\P(u\in\boldsymbol{B})} and by Proposition \ref{prop: I<=00003Dlogplam}}) & \le s\sum_{u\in\Vo}\logplam\P(\forall v\in N(u)\cap\boldsymbol{A},\boldsymbol{\ph}_{v}=0\condon u\in\boldsymbol{B})\\
 & =s\logplam\sum_{u\in\Vo}\E\left[\prod_{v\in N(u)\cap\boldsymbol{A}}(1-\boldsymbol{\ph}_{v})\condon u\in\boldsymbol{B}\right].
\end{align*}
Finally, for the third term, $\logplam\frac{|V|}{2}=\logplam\sum_{u\in\Vo}1$.

We turn to inequality (2), assuming the two uniform neighbors property.
The inequality will be shown for each $u\in\Vo$ separately. We prove
the inequality for a fixed $\ph$, using only the randomness in $(\boldsymbol{A},\boldsymbol{B})$.
This suffices since $(\boldsymbol{A},\boldsymbol{B})$ is independent
of $\boldsymbol{\ph}$.

Fix $u\in\Vo$ and $\ph$. Let $(\boldsymbol{v}_{u,1},\boldsymbol{v}_{u,2})$
be as in the two uniform neighbors property, so that conditioned on
$u\in\boldsymbol{B}$ they are a uniformly sampled pair of vertices
in $N(u)$ (not necessarily distinct). For brevity, denote $\E^{u}[\cdot]:=\E[\cdot\condon u\in B]$.
Then
\begin{align*}
\E^{u}\left[\prod_{v\in N(u)\cap\boldsymbol{A}}(1-\ph_{v})+\phm_{u}-1\right] & \le\E^{u}\left[(1-\ph_{\boldsymbol{v}_{u,1}})(1-\ph_{\boldsymbol{v}_{u,2}})+\ph_{\boldsymbol{v}_{u,1}}-1\right]\\
 & =\E^{u}\left[-\ph_{\boldsymbol{v}_{u,2}}(1-\ph_{\boldsymbol{v}_{u,1}})\right]\\
\left(\substack{\text{by the two uniform}\\
\text{neighbors property}
}
\right) & =-\phm_{u}(1-\phm_{u}).
\end{align*}

Finally, to see inequality (3). Observe that $x(1-x)\ge\frac{1}{2}\min\{x,1-x\}$
for $x\in[0,1].$ Therefore by (\ref{eq:Phi}),
\begin{align*}
\sum_{u\in\Vo}\phm_{u}(1-\phm_{u}) & \ge\frac{1}{2}\sum_{u\in\Vo}\min\{\phm_{u},1-\phm_{u}\}=\frac{1}{4}|V|\Phi.\qedhere
\end{align*}
\end{proof}

\section{The loss term\label{sec:The-loss-term}}

This section is devoted to the proof of the following Lemma \ref{lem:S<=00003DPhi},
and is the only place where the local expansion property (Definition
\ref{def:loc_exp}) is used directly.

Let $G=(V,E)$ be a finite $\delta$-regular graph (not necessarily
bipartite) satisfying the local expansion property with parameters
$C_{\LE},M_{\LE}$. Let $\boldsymbol{\ph}:V\to\{0,1\}$ be a random
assignment of bits to the vertices of $G$ (sampled from an arbitrary
distribution). Note that we do not assume that $\boldsymbol{\ph}$
is the one introduced in Section \ref{sec:phi}, but this notation
is suggestive, as the application of Lemma \ref{lem:S<=00003DPhi}
below is to the $\boldsymbol{\ph}$ of Section \ref{sec:phi}. We
let $\boldsymbol{\Phi}$ be given by (\ref{eq:Phi}). The following
lemma says that given local expansion, a random restriction of $\boldsymbol{\ph}$
may be ``encoded efficiently'' if $\boldsymbol{\ph}$ is sufficiently
``smooth''.
\begin{lem}
\label{lem:S<=00003DPhi}%
Let $\boldsymbol{A}$ be a random subset of $V$, independent of $\boldsymbol{\ph}$
and satisfying $\P(v\in\boldsymbol{A})\le p$ for every $v\in V$.
Then
\begin{equation}
S(\boldsymbol{\ph}_{\boldsymbol{A}}\condon\boldsymbol{A})\le C_{\LE}|V|\left(2(p+\frac{1}{\delta})\cdot S(\E\boldsymbol{\Phi})+\frac{\log2}{\delta M_{\LE}}\right).\label{eq:compress_bound}
\end{equation}
\end{lem}

\begin{proof}
[Proof of Lemma \ref{lem:S<=00003DPhi}]Let $\boldsymbol{T}$ be
the random subgraph given by the local expansion property, sampled
independently of $\boldsymbol{\ph}$ and $\boldsymbol{A}$. Denote
$N(T)\coloneqq\bigcup_{v\in V(T)}N(v)$. Let $A\subset V$. Apply
the generalized Shearer's inequality (Theorem \ref{thm:shearer_gen}
with $(\boldsymbol{X}_{j})_{j\in J}=\boldsymbol{\ph}_{A}$ and $\boldsymbol{K}=N(\boldsymbol{T})\cap A$),
in the probability space conditioned on $\boldsymbol{A}=A$. This
gives
\begin{equation}
S(\boldsymbol{\ph}_{A}\condon\boldsymbol{A}=A)\le\frac{S(\boldsymbol{\ph}_{N(\boldsymbol{T})\cap A}\condon\boldsymbol{T};\boldsymbol{A}=A)}{\min_{v\in A}\P(v\in N(\boldsymbol{T}))}.\label{eq:walk_shearer}
\end{equation}

To bound the (expectation over $A$ of the) numerator we introduce
the edge set
\[
E(\boldsymbol{T},\boldsymbol{A})\coloneqq E(\boldsymbol{T})\cup\{uv\in E:u\in\boldsymbol{A},v\in V(\boldsymbol{T})\}.
\]
Observe that by item 1 of the local expansion property $\P(uv\in E(\boldsymbol{T}))\le2M_{\LE}/(\delta|V|)$
for each $uv\in E$, and therefore also we have $\P(u\in V(\boldsymbol{T}))\le2M_{\LE}/|V|$
for each $u\in V$. the independence of $\boldsymbol{T}$ and $\boldsymbol{A}$,
and the assumption that $\max_{v}\P(v\in\boldsymbol{A})\le p$,
\begin{align}
 & \P\left(uv\in E(\boldsymbol{T},\boldsymbol{A})\right)\nonumber \\
 & \le\P(uv\in E(\boldsymbol{T}))+\P(u\in V(\boldsymbol{T}))\P(v\in\boldsymbol{A})+\P(u\in\boldsymbol{A})\P(v\in V(\boldsymbol{T}))\nonumber \\
 & \le\frac{4M_{\LE}}{|V|}(\frac{1}{\delta}+p).\label{eq:uv_in_K}
\end{align}
Note that $E(\boldsymbol{T},\boldsymbol{A})$ spans a connected graph,
whose vertex set contains $N(\boldsymbol{T})\cap\boldsymbol{A}$.
Since $\boldsymbol{\ph}$ takes values in $\{0,1\}$, it follows that
$\boldsymbol{\ph}_{N(\boldsymbol{T})\cap A}$ may be deduced from
$(|\boldsymbol{\ph}_{u}-\boldsymbol{\ph}_{v}|)_{uv\in E(\boldsymbol{T},\boldsymbol{A})}$
and $\boldsymbol{\ph}_{v_{0}(\boldsymbol{T})}$ where $v_{0}(\boldsymbol{T})$
denotes an arbitrary vertex of $\boldsymbol{T}$ designated as its
root. Thus, using that fact that $\boldsymbol{T}$ is independent
of $(\boldsymbol{\ph},\boldsymbol{A})$
\begin{align}
S(\boldsymbol{\ph}_{N(\boldsymbol{T})\cap\boldsymbol{A}}\condon\boldsymbol{T},\boldsymbol{A}) & \le S(\ph_{v_{0}(\boldsymbol{T})}\condon\boldsymbol{T},\boldsymbol{A})+\sum_{uv\in E}S(|\boldsymbol{\ph}_{u}-\boldsymbol{\ph}_{v}|\cdot\indic{uv\in E(\boldsymbol{T},\boldsymbol{A})}\condon\boldsymbol{T},\boldsymbol{A})\nonumber \\
 & =S(\ph_{v_{0}(\boldsymbol{T})}\condon\boldsymbol{T},\boldsymbol{A})+\sum_{uv\in E}\P\left(uv\in E(\boldsymbol{T},\boldsymbol{A})\right)S(|\boldsymbol{\ph}_{u}-\boldsymbol{\ph}_{v}|)\nonumber \\
(\text{by (\ref{eq:uv_in_K}), binary entropy}) & \le\log2+\frac{4M_{\LE}}{|V|}(\frac{1}{\delta}+p)\sum_{uv\in E}S(\E|\boldsymbol{\ph}_{u}-\boldsymbol{\ph}_{v}|)\nonumber \\
\text{(Jensen's inequality, \ensuremath{2|E|=|V|\delta})} & \le\log2+2\delta M_{\LE}(\frac{1}{\delta}+p)S(\E\boldsymbol{\Phi}).\label{eq:walk_num_bound}
\end{align}

Finally,
\begin{align*}
S(\boldsymbol{\ph}_{A}\condon\boldsymbol{A}) & =\sum_{A\subset V}\P(\boldsymbol{A}=A)S(\boldsymbol{\ph}_{A}\condon\boldsymbol{A}=A)\\
(\text{by (\ref{eq:walk_shearer}), and \ensuremath{A\subset V}}) & \le\sum_{A\subset V}\P(\boldsymbol{A}=A)\frac{S(\boldsymbol{\ph}_{N(\boldsymbol{T})\cap A}\condon\boldsymbol{T};\boldsymbol{A}=A)}{\min_{v\in A}\P(v\in N(\boldsymbol{T}))}\\
 & =\frac{S(\boldsymbol{\ph}_{N(\boldsymbol{T})\cap A}\condon\boldsymbol{T},\boldsymbol{A})}{\min_{v\in V}\P(v\in N(\boldsymbol{T}))}
\end{align*}
and the lemma follows by substituting (\ref{eq:walk_num_bound}) and
item \ref{enu:loc_exp_covering} of the local expansion property (noting
that $v\in N(T)\iff N(v)\cap V(T)\neq\emptyset$).%
\end{proof}

\section{\label{sec:Proof-of-Proposition}Proof of Proposition \ref{prop:I<=00003D-Phi}\protect 
}\begin{proof}
Let $G$ be a finite $\delta$-regular bipartite graph, satisfying
the local expansion property with parameters \emph{$C_{\LE},M_{\LE}$}.
Let $\boldsymbol{\sigma}$ be sampled from \emph{any measure} on $\Omega_{\hc}^{G}$.
Let $\lambda>0$. Let $\boldsymbol{A}\subset\Ve$ be random, sampled
independently of $\boldsymbol{\sigma}$, with $\P(v\in\boldsymbol{A})=1/\delta$
independently for each $v\in\Ve$. Define $\boldsymbol{B}\subset\Vo$
to be $\{u\in\Vo:|N(u)\cap\boldsymbol{A}|\ge2\}$. Then the pair $(\boldsymbol{A},\boldsymbol{B})$
satisfies the assumptions of Lemma \ref{lem:three term decomposition}
with 
\[
s=\P(u\in\boldsymbol{B})=1-(1-1/\delta)^{\delta}-\delta\frac{1}{\delta}(1-1/\delta)^{\delta-1}\ge c
\]
for all $u\in\Vo$ (using that $\delta\ge2$). Lemma \ref{lem: gain terms estimate}
applies to $(\boldsymbol{A},\boldsymbol{B})$ as well, since $(\boldsymbol{A},\boldsymbol{B})$
satisfies the two uniform neighbors property. Notice that $\boldsymbol{\ph},\boldsymbol{\Phi},\boldsymbol{A},p=\frac{1}{\delta}$
satisfy the conditions of Lemma \ref{lem:S<=00003DPhi}, and thus
(\ref{eq:compress_bound}) holds. We now combine the three-term decomposition
with the bounds on the gain terms and the loss term:
\begin{align*}
 & I(\boldsymbol{\sigma})-\logplam\frac{|V|}{2}\\
\text{(by Lemma \ref{lem:three term decomposition})} & \le I(\boldsymbol{\sigma}_{\Ve}\condon\boldsymbol{\sigma}_{\Vo})+\frac{1}{s}I(\boldsymbol{\sigma}_{\boldsymbol{B}}\condon\boldsymbol{B},\boldsymbol{A},\boldsymbol{\ph}_{\boldsymbol{A}})-\logplam\frac{|V|}{2}+\frac{1}{s}S(\boldsymbol{\ph}_{\boldsymbol{A}}\condon\boldsymbol{A}).\\
\text{(Lemma \ref{lem: gain terms estimate}, and \ensuremath{1/s\le C})} & \le-\frac{1}{4}\logplam|V|\E\boldsymbol{\Phi}+C\cdot S(\boldsymbol{\ph}_{\boldsymbol{A}}\condon\boldsymbol{A})\\
\text{(by Lemma \ref{lem:S<=00003DPhi} and (\ref{eq:binary_entropy_bound}))} & \le-\frac{1}{4}\logplam|V|\E\boldsymbol{\Phi}+CC_{\LE}|V|\left(\frac{4}{\delta}\E\boldsymbol{\Phi}\log\frac{e}{\E\boldsymbol{\Phi}}+\frac{\log2}{\delta M_{\LE}}\right)
\end{align*}
which verifies (\ref{eq:I<=00003DPhi bound}) (using a larger value
of $C$). The second part of the proposition is an immediate consequence.
\end{proof}

\subsection{Simplified proof in a special case\label{subsec:Simplified-proof}}

The proof of Proposition \ref{prop:I<=00003D-Phi} (which is the main
ingredient in the proof of Theorem \ref{thm:main_finite}) admits
a substantial simplification in a special case, as we now describe.

A dominating tree of a graph is a subgraph which is a tree and whose
vertex set is dominating (see Section \ref{sec:Local-expansion-for}
below for the definition). Let $q\in(0,1]$ and assume that the graph
$G$ has a random dominating tree $\boldsymbol{T}$ (having at least
one edge) with 
\begin{equation}
\max_{uv\in E}\P(uv\in E(\boldsymbol{T}))\le q.\label{eq:edge_uniformity}
\end{equation}
Using this $\boldsymbol{T}$ in the definition of the local expansion
property, we see that the property is satisfied with $M_{\LE}=q|E|$
and $C_{\LE}=\frac{\delta M_{\LE}}{|V|}$ (in fact, in the context
of Definition \ref{def:loc_exp}, the last equality is equivalent
to $V(\boldsymbol{T})$ being almost surely dominating). The proof
of Proposition \ref{prop:I<=00003D-Phi} under the above assumption
(and with these resulting $C_{\LE},M_{\LE}$ parameters) may be simplified
as follows:
\begin{enumerate}
\item Restrict Lemma \ref{lem:three term decomposition} and Lemma \ref{lem: gain terms estimate}
to the case that $\boldsymbol{B}=\Vo$ deterministically (and thus
$s=1$). This circumvents the use of Shearer's inequality in Lemma
\ref{lem:three term decomposition} and the conditioning on $u\in\boldsymbol{B}$
in Lemma \ref{lem: gain terms estimate}.
\item Replace Lemma \ref{lem:S<=00003DPhi} by the statement $S(\boldsymbol{\ph}_{V(\boldsymbol{T})}\condon\boldsymbol{T})\le q|E|\cdot S(\E\boldsymbol{\Phi})+\log2$,
where $\boldsymbol{T}$ is any random tree satisfying (\ref{eq:edge_uniformity}),
sampled independently of $\boldsymbol{\ph}$. The proof is a variation
on (\ref{eq:walk_num_bound}) where we put $\boldsymbol{T}$ instead
of $\boldsymbol{T},\boldsymbol{A}$, and $V(\boldsymbol{T})$ instead
of $N(\boldsymbol{T})\cap\boldsymbol{A}$, and we use (\ref{eq:edge_uniformity})
instead of (\ref{eq:uv_in_K}).
\end{enumerate}
To complete the proof of Proposition \ref{prop:I<=00003D-Phi} in
the special case, we modify the choice of $\boldsymbol{A}$ in Section
\ref{sec:Proof-of-Proposition}, by taking $\boldsymbol{A}=V(\boldsymbol{T}_{1})\cup V(\boldsymbol{T}_{2})$
with $\boldsymbol{T}_{1},\boldsymbol{T}_{2}$ being two independent
copies of the random dominating tree $\boldsymbol{T}$. This leads
to $\boldsymbol{B}=\Vo$. We then check that the restricted Lemma
\ref{lem:three term decomposition} and Lemma \ref{lem: gain terms estimate}
apply to this alternative choice of $(\boldsymbol{A},\boldsymbol{B})$
and bound $S(\boldsymbol{\ph}_{\boldsymbol{A}}\condon\boldsymbol{A})$
by noting that $S(\boldsymbol{\ph}_{\boldsymbol{A}}\condon\boldsymbol{A})\le2S(\boldsymbol{\ph}_{V(\boldsymbol{T})}\condon\boldsymbol{T})$
and using the restricted Lemma \ref{lem:S<=00003DPhi}.

As an important example, we point out that $\Z_{L}^{d}$ has a random
dominating tree as above with $q\le\frac{C}{d^{2}}$ by the construction
of Section \ref{sec:Local-expansion-for} below. Thus, the above special
case of Proposition \ref{prop:I<=00003D-Phi} suffices when proving
Corollary \ref{cor:torus}, Corollary \ref{cor:Zd_SFE}, and consequently
Theorem \ref{thm: two Gibbs measures} (but not Theorem \ref{thm:main_finite}
in its full generality). In particular, this results in a proof of
these 3 results which does not use Shearer's inequality.

\part{Long-range order on discrete tori\label{part:torus}}

In this part, we prove Lemma \ref{lem: expansion for torus graphs}
(Section \ref{sec:Local-expansion-for} and Section \ref{sec:Global-expansion-for})
and Corollary \ref{cor:torus} (Section \ref{sec:proof_of_cor_torus}).
Corollary \ref{cor:Zd_SFE} is an immediate consequence of Theorem
\ref{thm:main_finite} by substituting the parameters of Lemma \ref{lem: expansion for torus graphs}.

\section{Local expansion for $\protect\Z_{L}^{d}$\label{sec:Local-expansion-for}}

Let $G$ be a graph. For $A\subset V$ denote by $A^{*}$ the set
of vertices that are in $A$ or have a neighbor in $A$. A \emph{dominating
set} is a subset $D\subset V$, with $D^{*}=V$. Define the \emph{domination
number} as
\[
\gamma(G):=\min\{|D|:D\subset V,D^{*}=V\}.
\]

When a subgraph $T$ of $G$ is a tree and $V(T)$ is dominating in
$G$, we call $T$ a \emph{dominating tree} of $G$.
\begin{lem}
\label{lem:Z_L^d_dominating}Let $d\ge1,L\ge2$ integers. Then $\gamma(\Z_{L}^{d})<2L^{d}/d$.
\end{lem}

\begin{proof}
Set $r:=\max\{s\in\N:2^{s}-1\le d\}$ and denote $d'=2^{r}-1$. Let
$H\subset\Z_{2}^{d'}$ be the Hamming code \cite{enwiki:1319758292}.
Since $H$ is a perfect, single error-correcting code, it is a dominating
set and $|H|=2^{d'}/(d'+1)$%
. In the case that $L$ is even, we may simply choose
\[
D\coloneqq\left\{ (v_{1},\dots,v_{d})\in\Z_{L}^{d}:(v_{1}\bmod2,\dots,v_{d'}\bmod2)\in H\right\} 
\]
as a dominating set for $\Z_{L}^{d}$, noting that $|D|=L^{d}/(d'+1)<2L^{d}/d$.
When $L$ is odd, denote
\[
D_{b}\coloneqq\left\{ (v_{1},\dots,v_{d})\in\{0,\dots,L-1\}^{d}:(v_{1}\bmod2,\dots,v_{d'}\bmod2)\in H+b\right\} 
\]
where $b\in\Z_{2}^{d'}$. Note that here $v_{1},\dots,v_{d}$ are
taken to be integers rather than residue classes modulo $L$, so that
they can be considered modulo $2$. The set $D_{b}$ viewed as a subset
of $\Z_{L}^{d}$ is dominating for all $b\in\Z_{2}^{d'}$. Taking
$\boldsymbol{b}\in\Z_{2}^{d'}$ uniformly random, it holds that
\[
\E[|D_{\boldsymbol{b}}|]=|\Z_{L}^{d}|\cdot|H|/|\Z_{2}^{d'}|=L^{d}/(d'+1)<2L^{d}/d.
\]
Thus there is a choice of $b$ such that $|D_{b}|<2L^{d}/d$.
\end{proof}
\begin{lem}
\label{lem:dominating_tree}Let $G$ be a connected graph. Then $G$
has a dominating tree $T$ of order at most $3\gamma(G)$.
\end{lem}

\begin{proof}
Let $D$ be a dominating set in $G$ with $|D|=\gamma(G)$. Construct
a graph $G_{0}$ with vertex set $D$ where two vertices $u,v\in D$
are adjacent in $G_{0}$ iff the distance from $u$ to $v$ in $G$
is at most $3$. We claim that $G_{0}$ is connected. Indeed, assume
for contradiction that $A$ is the vertex set of a connected component,
and that $B\coloneqq D\setminus A$ is non-empty. Then since $D$
is a dominating set, $A^{*}\cup B^{*}=G$ (where $*$ is taken with
respect to the connectivity of $G$). By the connectedness of $G$
and the fact that $A^{*}$, $B^{*}$ are nonempty it follows that
there is an edge $uv\in E(G)$ with $u\in A^{*}$,$v\in B^{*}$. Thus
there is $u'\in A$ with $u$ being equal or adjacent to $u'$ and
similarly $v'\in B$ that is equal or adjacent to $v$. Thus $u'$,$v'$
are at graph distance at most $3$ from each other in $G$, so that
$u'v'\in E(G_{0})$ contradicting the assumption.

We construct $T$ as follows. Take a spanning tree of $G_{0}$ and
for each edge $uv$ of it, pick a path of length at most $3$ in $G$
with endpoints $u$ and $v$. Denote by $T'$ the connected subgraph
of $G$ formed by the union of these paths. Note that $|V(T')|\le3\gamma(G)$
and $D\subset V(T')$. Take $T$ to be an arbitrary spanning tree
of $T'$.
\end{proof}
\begin{lem}
\label{lem:transitive+dom->loc_exp}Let $G$ be an edge transitive
graph, having a dominating tree of order $M$ with at least one edge.
Then $G$ satisfies the local expansion property with any $M_{\LE}\ge M$
and $C_{\LE}\ge\frac{\delta M_{\LE}}{|V|}$.
\end{lem}

\begin{proof}
Let $T$ be a dominating tree of order $M$. Define a random connected
subgraph $\boldsymbol{T}\subset G$ by acting on $T$ by a random
automorphism of $G$. Since $G$ is edge transitive, it follows that
item 1 is satisfied as $|E(T)|=M-1\le M_{\LE}$. For item 2 the probability
is $1$, since $V(\boldsymbol{T})$ is always dominating, and thus
the condition holds with any $C_{\LE}\ge\frac{\delta M_{\LE}}{|V|}$.
\end{proof}
Let $d\ge1,L\ge2$ integers. Combining Lemma \ref{lem:Z_L^d_dominating}
and Lemma \ref{lem:dominating_tree} above, $\Z_{L}^{d}$ has a dominating
tree of size at most $6L^{d}/d$. Since $\Z_{L}^{d}$ is edge transitive,
Lemma \ref{lem:transitive+dom->loc_exp} implies that $\Z_{L}^{d}$
satisfies the local expansion property with $M_{\LE}=6L^{d}/d$ and
$C_{\LE}=\frac{2d\cdot6L^{d}/d}{L^{d}}=12$ (we have $\delta=2d$
for $L\ge3$ and $\delta=d$ when $L=2$). This gives the local expansion
part of Lemma \ref{lem: expansion for torus graphs}.

\section{Global expansion for $\protect\Z_{L}^{d}$\label{sec:Global-expansion-for}}

The following lemma gives the global expansion part of Lemma \ref{lem: expansion for torus graphs}.
\begin{lem}
\label{lem:torus_cheeger}$h(\Z_{L}^{d})\ge\frac{1}{L}$ for integers
$d\ge1,L\ge2$.
\end{lem}

\begin{proof}
The one-dimensional case is trivial: $h(\Z_{L})\ge\frac{4}{L}$ when
$L\ge3$ and $h(\Z_{L})=1$ when $L=2$. To go up in dimension, we
use the fact that for every graph $G$, it holds that $h(G)/2\le h(G^{d})\le h(G)$;
see \cite[Theorem 1.2]{chung1998isoperimetric} or \cite[Section 3, Proposition 1]{tillich2000edge}.
\end{proof}
\begin{rem}
The Poincaré inequality gives the weaker bound $h(\Z_{L}^{d})\ge\frac{c}{L^{2}}$.
We note that our proof of Theorem \ref{thm: two Gibbs measures} only
relies on the case $L=6$, making use of the inequality $h(\Z_{6}^{d})\ge c$,
for which the weaker bound is also sufficient.%
\end{rem}

\section{\label{sec:proof_of_cor_torus}Proof of Corollary \ref{cor:torus}}

We will use the following lemma regarding the tails of the binomial
distribution.
\begin{lem}
\label{lem:hoeffding}For every $n\in\N$, $m\ge0$, and $0<p<1$,
\[
\P\left(\left|\mathrm{Bin}(n,p)-np\right|\ge m\right)\le2\left(p(1-p)\right)^{m^{2}/n}.
\]
\end{lem}

\begin{proof}
We have
\[
\P\left(\mathrm{Bin}(n,p)-np\ge m\right)\le\exp\left(-\frac{m^{2}}{n}g(p)\right)\le\left(p(1-p)\right)^{m^{2}/n}.
\]
The first inequality is \cite[(2.2)]{hoeffdingProbabilityInequalitiesSums1963},
with the function $g(p):=\begin{cases}
\frac{1}{1-2p}\log\frac{1-p}{p} & 0<p<\frac{1}{2}\\
\frac{1}{2p(1-p)} & \frac{1}{2}\le p<1
\end{cases}$ defined in \cite[(2.4)]{hoeffdingProbabilityInequalitiesSums1963}.
The second inequality follows by noting (with elementary calculus)
that
\[
g(p)>\log\frac{1}{p(1-p)}.
\]
By symmetry of our lower bound for $g(p)$, the same bound applies
for
\[
\P\left(\mathrm{Bin}(n,p)-np\le-m\right)=\P\left(\mathrm{Bin}(n,1-p)-n(1-p)\ge m\right).\qedhere
\]
\end{proof}
\begin{proof}
[Proof of Corollary \ref{cor:torus}]The first inequality is simply
(\ref{eq:trivial partition function lower bound}). For the second
inequality we begin by bounding $\mu^{\mathbb{Z}_{L}^{d}}\left(B\right)$
and then deduce the required bound on $\zeta^{\mathbb{Z}_{L}^{d}}\left(B\right)$.
Let $C_{0}>0$. Write $G_{\Ve}:=\left\{ |\sigma_{\Vo}|\le\frac{L^{d+1}}{4d^{C_{0}+1}}\right\} $,
$G_{\Vo}:=\left\{ |\sigma_{\Ve}|\le\frac{L^{d+1}}{4d^{C_{0}+1}}\right\} $
and $B_{0}:=\left\{ \left||\sigma|-\frac{L^{d}}{2}\frac{\lambda}{1+\lambda}\right|>\frac{L^{d+1}}{d^{C_{0}}}\right\} $.
Note that for $B$ from (\ref{eq:discrete tori bad event}) we have
\begin{equation}
B\subset(G_{\Ve}\cup G_{\Vo})^{c}\cup B_{0}=(G_{\Ve}\cup G_{\Vo})^{c}\cup(B_{0}\cap G_{\Ve})\cup(B_{0}\cap G_{\Vo}).\label{eq:torus bad decomposition}
\end{equation}
We now bound the probability of $(G_{\Ve}\cup G_{\Vo})^{c}$. Set
$r=\frac{L^{d+1}}{4d^{C_{0}+1}}$ and note that (\ref{eq:main_ass})
follows from the assumption $\lambda>C_{1}\frac{\log d}{d}$ (for
sufficiently large $C_{1}$) by Lemma \ref{lem: expansion for torus graphs}.
Thus (\ref{eq:main_bound}) holds, giving
\[
\mu^{\mathbb{Z}_{L}^{d}}((G_{\Ve}\cup G_{\Vo})^{c})\le(1+\lambda)^{-c\frac{L^{d}}{d^{C_{0}+2}}}.
\]

We proceed to bound the probability of $B_{0}\cap G_{\Ve}$. Conditioned
on $\boldsymbol{\sigma}_{\Vo}$, the variable $|\boldsymbol{\sigma}_{\Ve}|$
is distributed as $\mathrm{Bin}(|\boldsymbol{\ph}_{\Ve}|,\frac{\lambda}{1+\lambda})$
with $\ph$ defined in (\ref{eq:phi}). Note that $\frac{L^{d}}{2}-2d|\sigma_{\Vo}|\le|\ph_{\Ve}|\le\frac{L^{d}}{2}$
(since $\mathbb{Z}_{L}^{d}$ has degree bound $2d$). Now, since
\begin{align*}
B_{0}\cap G_{\Ve} & =\left\{ \left||\boldsymbol{\sigma}|-\frac{L^{d}}{2}\frac{\lambda}{1+\lambda}\right|>\frac{L^{d+1}}{d^{C_{0}}}\right\} \cap G_{\Ve}\\
 & =\left\{ \left||\boldsymbol{\sigma}_{\Ve}|-|\boldsymbol{\ph}_{\Ve}|\frac{\lambda}{1+\lambda}+|\boldsymbol{\sigma}_{\Vo}|-(\frac{L^{d}}{2}-|\boldsymbol{\ph}_{\Ve}|)\frac{\lambda}{1+\lambda}\right|>\frac{L^{d+1}}{d^{C_{0}}}\right\} \cap G_{\Ve}\\
 & \subset\left\{ \left||\boldsymbol{\sigma}_{\Ve}|-|\boldsymbol{\ph}_{\Ve}|\frac{\lambda}{1+\lambda}\right|>\frac{L^{d+1}}{d^{C_{0}}}-\left||\boldsymbol{\sigma}_{\Vo}|-(\frac{L^{d}}{2}-|\boldsymbol{\ph}_{\Ve}|)\frac{\lambda}{1+\lambda}\right|\right\} \cap G_{\Ve}\\
(|\ph_{\Ve}|\le\frac{L^{d}}{2}-2d|\sigma_{\Vo}|,\text{\ensuremath{\frac{\lambda}{1+\lambda}\le1})} & \subset\left\{ \left||\boldsymbol{\sigma}_{\Ve}|-|\boldsymbol{\ph}_{\Ve}|\frac{\lambda}{1+\lambda}\right|>\frac{L^{d+1}}{d^{C_{0}}}-(2d-1)\frac{L^{d+1}}{4d^{C_{0}+1}}\right\} \\
 & \subset\left\{ \left||\boldsymbol{\sigma}_{\Ve}|-|\boldsymbol{\ph}_{\Ve}|\frac{\lambda}{1+\lambda}\right|>\frac{L^{d+1}}{2d^{C_{0}}}\right\} 
\end{align*}

we conclude, using Lemma \ref{lem:hoeffding} with $n=|\boldsymbol{\ph}_{\Ve}|$,
$p=\frac{\lambda}{1+\lambda}$ and $m=\frac{L^{d+1}}{2d^{C_{0}}}$,
that
\[
\mu^{\mathbb{Z}_{L}^{d}}(B_{0}\cap G_{\Ve})\le2\max_{n\le\frac{L^{d}}{2}}\left(p(1-p)\right)^{m^{2}/n}\le2\max_{n\le\frac{L^{d}}{2}}\left(1-p\right)^{m^{2}/n}=2\left(1+\lambda\right)^{-\frac{2}{L^{d}}\left(\frac{L^{d+1}}{2d^{C_{0}}}\right)^{2}}=2\left(1+\lambda\right)^{-\frac{L^{d+2}}{2d^{2C_{0}}}}.
\]
Analogously, we get the same bound for the probability of $B_{0}\cap G_{\Vo}$.
Combining all the bounds obtained so far, we get $\mu^{\mathbb{Z}_{L}^{d}}\left(B\right)\le5(1+\lambda)^{-c\frac{L^{d}}{d^{2C_{0}+2}}}$.
By Corollary \ref{cor:Zd_SFE} together with $\lambda>C_{1}\frac{\log d}{d}$,
we have
\begin{align*}
\frac{Z^{\mathbb{Z}_{L}^{d}}}{(1+\lambda)^{L^{d}/2}} & \le C\exp\left(L^{d}\frac{C}{d}(1+\frac{C_{1}\log d}{d})^{-cd}\right)\\
 & =C(1+\lambda)^{CL^{d}\frac{1}{\logplam d}(1+\frac{C_{1}\log d}{d})^{-cd}}\\
 & \le C\left(1+\lambda\right)^{CL^{d}(1+\frac{C_{1}\log d}{d})^{-cd}}.
\end{align*}
Thus,
\[
\frac{\zeta^{\mathbb{Z}_{L}^{d}}\left(B\right)}{(1+\lambda)^{L^{d}/2}}=\mu^{\mathbb{Z}_{L}^{d}}\left(B\right)\frac{Z^{\mathbb{Z}_{L}^{d}}}{(1+\lambda)^{L^{d}/2}}\le5C(1+\lambda)^{-c\frac{L^{d}}{d^{2C_{0}+2}}+CL^{d}(1+\frac{C_{1}\log d}{d})^{-cd}}\le5C(1+\lambda)^{-c\frac{L^{d}}{d^{2C_{0}+2}}}\le(1+\lambda)^{-\frac{L^{d}}{d^{C_{2}}}}
\]

where the second-to-last inequality uses that $C_{1}$ is large enough
as a function of $C_{0}$ and the last inequality follows by taking
$C_{1},C_{2}$ large enough as a function of $C_{0}$ and using that
$\lambda>C_{1}\frac{\log d}{d}$ and $L,d\ge2$.%
\end{proof}

\part{Long-range order on $\protect\Z^{d}$\label{part:Long-range-order-on}}

Our proof of the existence of multiple Gibbs measures (Theorem \ref{thm: two Gibbs measures})
is an instance of the chessboard Peierls argument (see \cite{frohlichPhaseTransitionsReflection1978,frohlichPhaseTransitionsReflection1980}
where this term is introduced, to the best of our knowledge). We define
contours based on a grid of cubes of side length $\ell=3$, and use
the chessboard estimate to relate the probability for the appearance
of a contour to the probability of a fixed event in the hard-core
model on $\Z_{6}^{d}$, which we bound via Theorem \ref{thm:main_finite}.

\section{The chessboard estimate\label{sec:Chessboard}}

In this section, we set up the notation and properties required for
our use of the chessboard estimate. Our presentation adapts the notation
of \cite{hadasColumnarOrderRandom2022}.

The chessboard estimate is a consequence of the reflection positivity
of the hard-core measure, with respect to reflections through coordinate
hyperplanes \emph{passing through vertices}. We do not give a proof
of the chessboard estimate or of the reflection positivity property
and refer to \cite{frohlichPhaseTransitionsReflection1978}, \cite{shlosman1986method},
\cite{biskup2009reflection}, \cite[Chapter 10]{friedli_velenik_2017},
\cite[Section 2.7.1]{PeledSpinka2019}, \cite{hadas2022chessboard}
for pedagogical references. We note, however, that they are a consequence
of the fact that the interactions of the hard-core model are nearest-neighbor
symmetric pair interactions (or more generally, reflection-invariant
hypercube interactions) and the results of this section generalize
immediately to other models with such interactions.

\subsection{Definitions}

To state the chessboard estimate, we make the following definitions,
for each $\ell\in\N$, $L\in2\ell\N$ and $\lambda>0$:
\begin{itemize}
\item Recall that $\Z_{m}=\Z/m\Z$ and that $\Z_{m}^{d}$ stands for both
a set of vectors of residue classes and a torus graph, as defined
in Subsection \ref{subsec: results on finite graphs}.
\item Let $\refls$ denote the dihedral group of $L/\ell$ elements. We
consider its action on $\R/L\Z$ where rotations, indexed by $n\in\mathbb{Z}_{L/2\ell}$,
act as $\tau v=v+2n\ell$ and reflections, also indexed by $n\in\mathbb{Z}_{L/2\ell}$,
act as $\tau v=2n\ell-v$.
\begin{itemize}
\item The direct product of $d$ copies, $\refls^{d}$, acts on the Cartesian
product $(\R/L\Z)^{d}$ in the natural way. Observe that $[0,\ell]^{d}$
is a fundamental domain for this action.
\item Define the action of $\tau\in\refls^{d}$ on $\{0,1\}^{\mathbb{Z}_{L}^{d}}$
and its action on $\R^{\{0,1\}^{\mathbb{Z}_{L}^{d}}}$ by 
\begin{equation}
(\sigma\tau)_{v}\coloneqq\sigma_{\tau v}\quad\text{and}\quad(\tau f)(\sigma)\coloneqq f(\sigma\tau).\label{eq:actions}
\end{equation}
\end{itemize}
\item A function $f:\{0,1\}^{\mathbb{Z}^{d}}\to\R$ is called $[0,\ell]^{d}$-local
if $f(\sigma)$ depends only on the restriction of $\sigma$ to $[0,\ell]^{d}\cap\Z^{d}$.
\begin{itemize}
\item We also regard such $f$ as $f:\{0,1\}^{\mathbb{Z}_{L}^{d}}\to\R$
by regarding $\sigma\in\{0,1\}^{\mathbb{Z}_{L}^{d}}$ as $L\Z^{d}$-periodic
functions on $\Z^{d}$.
\end{itemize}
\item Define the $(\ell,L)$-\textbf{chessboard seminorm} of a $[0,\ell]^{d}$-local
function $f$ by
\begin{align}
\zedd[L][\ell][f] & \coloneqq\sqrt[|\refls^{d}|]{\zeta^{\mathbb{Z}_{L}^{d}}\left(\prod_{\tau\in\refls^{d}}\tau f\right)}.\label{eq:chessboard norm def}
\end{align}

\begin{itemize}
\item We find it convenient to make this definition with the non-normalized
measure $\zeta^{\mathbb{Z}_{L}^{d}}$ (see (\ref{eq: zeta measure}))
instead of the (more common) normalized measure $\mu^{\mathbb{Z}_{L}^{d}}$.
This is especially useful for Lemma \ref{lem:chessboard_monotonicity}
below.
\item The integral in (\ref{eq:chessboard norm def}) is necessarily non-negative
by reflection positivity, so that $\zedd[L][\ell][f]$ is well defined
and satisfies 
\begin{equation}
\zedd[L][\ell][f]\ge0.\label{eq:non-negativity of zeta}
\end{equation}
\item The name `seminorm' is justified by Proposition \ref{prop:seminorm}
below.
\end{itemize}
\end{itemize}

\subsection{Properties}
\begin{prop}
[Chessboard estimate]\label{prop:chessboard} Let $\ell\in\N$,
$L\in2\ell\N$ and $\lambda>0$. For each tuple $(f_{\tau})_{\tau\in\refls}$
of $[0,\ell]^{d}$-local functions it holds that
\[
\zeta^{\mathbb{Z}_{L}^{d}}\left(\prod_{\tau\in\refls^{d}}\tau f_{\tau}\right)\le\prod_{\tau\in\refls^{d}}\zedd[L][\ell][f_{\tau}].
\]
\end{prop}

The next fact captures several basic properties of the chessboard
seminorm $\zedd[L][\ell][\cdot]$ (in particular, justifying the name
seminorm); see \cite[Proposition 3.3]{hadasColumnarOrderRandom2022},
or \cite[Lemma 5.9]{biskup2009reflection}, for a proof.
\begin{fact}
[Positive homogeneity, triangle inequality and monotonicity]\label{prop:seminorm}Let
$\ell\in\N$, $L\in2\ell\N$ and $\lambda>0$. The mapping $f\mapsto\zedd[L][\ell][f]$,
where $f$ ranges over $[0,\ell]^{d}$-local functions, satisfies
the following properties:
\end{fact}

\begin{enumerate}
\item Homogeneity: $\zedd[L][\ell][\alpha f]=\left|\alpha\right|\zedd[L][\ell][f]$
for $\alpha\in\R$.
\item Triangle inequality: $\zedd[L][\ell][f_{0}+f_{1}]\le\zedd[L][\ell][f_{0}]+\zedd[L][\ell][f_{1}]$.
\item Monotonicity: $\zedd[L][\ell][g]\ge\zedd[L][\ell][f]$ whenever $g\ge f\ge0.$
\end{enumerate}
The next lemma allows to compare the chessboard seminorm between tori
of different side lengths, when these side lengths are divisible by
one another. We would like to highlight this lemma as it is useful
in our approach and we have not seen it in the standard references
on reflection positivity. We arrived at it from a similar statement
in \cite[equation (6)]{chanUpperBoundsGrowth2015}, which refers to
\cite[see equation (8)]{calkinIndepSets1998}. We again emphasize
that its proof relies only on the fact that the interactions of the
hard-core model are nearest-neighbor symmetric pair interactions (or
more generally, reflection-invariant hypercube interactions).
\begin{lem}
[Comparing chessboard seminorms on different tori]\label{lem:chessboard_monotonicity}
Let $\ell\in\N$, $L\in2\ell\N$ and $\lambda>0$. For each $[0,\ell]^{d}$-local
function $f$, $\zedd[L+2\ell][\ell]\le\zedd[L][\ell]$.
\end{lem}

\begin{proof}
We use the fact that when $A$ is a positive semi-definite matrix
$\sqrt[k]{\trace(A^{k})}\ge\sqrt[k+1]{\trace(A^{k+1})}$ for each
$k\in\N$. Consider the sequence
\[
n_{i}=\sqrt[\frac{L^{d-i}(L+2\ell)^{i}}{\ell^{d}}]{\zeta^{\mathbb{Z}_{L}^{d-i}\times\mathbb{Z}_{L+2\ell}^{i}}\left(\prod_{\tau\in\refls^{d-i}\times\refls[][(L+2\ell)]^{i}}\tau f\right)}
\]

and note that $\zedd[L][\ell][f]=n_{0}$ and $\zedd[L+2\ell][\ell][f]=n_{d}$.
Thus it suffices to show that $n_{i}$ is decreasing with $i$. To
simplify the presentation we will only illustrate this by proving
that $n_{d}\le n_{d-1}$. We will define a suitable transfer matrix
$B$ such that
\begin{align*}
\zeta^{\mathbb{Z}_{L+2\ell}^{d}}\left(\prod_{\tau\in\refls[][(L+2\ell)L]^{d}}\tau f\right) & =\trace\left((BB^{T})^{\frac{L+2\ell}{2\ell}}\right)\\
\zeta^{\mathbb{Z}_{L}^{1}\times\mathbb{Z}_{L+2\ell}^{d-1}}\left(\prod_{\tau\in\refls[][L]^{1}\times\refls[][(L+2\ell)]^{d-1}}\tau f\right) & =\trace\left((BB^{T})^{\frac{L}{2\ell}}\right).
\end{align*}
Putting $A=BB^{T}$ and $k=\frac{L}{2\ell}$ gives $n_{d}\le n_{d-1}$.
The matrix $B$ is indexed by $\Omega_{\hc}^{\mathbb{Z}_{L+2\ell}^{d-1}}$
(configurations on hyperplanes). For $\sigma^{0},\sigma^{1}\in\Omega_{\hc}^{\mathbb{Z}_{L+2\ell}^{d-1}}$,
define
\[
\Xi(\sigma^{0},\sigma^{1})\coloneqq\{\sigma\in\Omega_{\hc}^{P_{\ell}\times\mathbb{Z}_{L+2\ell}^{d-1}}:\sigma_{\{0\}\times\mathbb{Z}_{L+2\ell}^{d-1}}=\sigma^{0},\sigma_{\{\ell\}\times\mathbb{Z}_{L+2\ell}^{d-1}}=\sigma^{1}\}
\]
\[
B(\sigma^{0},\sigma^{1})\coloneqq\int_{\Xi(\sigma^{0},\sigma^{1})}\lambda^{-\frac{|\sigma^{0}|+|\sigma^{1}|}{2}}\prod_{\tau\in\{\mathrm{Id}\}\times\refls[][(L+2\ell)]^{d-1}}\tau f(\sigma)\,d\zeta^{P_{\ell}\times\mathbb{Z}_{L+2\ell}^{d-1}}(\sigma)
\]
where $P_{\ell}$ is the path graph with vertex set $\{0,\dots,\ell\}$
connected in order.
\end{proof}

\section{\label{sec:long range}Long-range order on $\protect\Z^{d}$}

In this section we prove Theorem \ref{thm: two Gibbs measures}, showing
the existence of at least two distinct Gibbs measures for the hard-core
model on the lattice $\Z^{d}$ at fugacity $\lambda>C\frac{\log d}{d}$.

The theorem follows from the next lemma, in which we claim the existence
of a ``symmetry-breaking local observable'' $f$. The precise definition
of $f$ is given in Subsection \ref{subsec:def_of_nu}, but one should
think of $f=1$ ($f=-1$) as indicating that $\sigma$ is in the ``odd-(even-)occupied
phase'' in the cube $[0,\ell]^{d}$, while the complementary event
$f=0$ indicates that $[0,\ell]^{d}$ is part of a ``contour'' separating
the two phases. Working on $\Z_{L}^{d}$, this observable is mapped
to other cubes $([0,\ell]^{d}+\ell s)_{s\in\Z_{L/\ell}^{d}}$ via
the mappings $\tau\in\refls^{d}$. The lemma quantifies the rarity
of contours at high fugacities. Its proof is enabled by the chessboard
estimate and the result of Theorem \ref{thm:main_finite} for the
discrete torus graph $\mathbb{\Z}_{2\ell}^{d}$.

Introduce the following notation: For each $s\in\Z_{L/\ell}^{d}$
denote by $\tau_{s}\in\refls^{d}$ the unique element which maps the
fundamental domain $[0,\ell]^{d}$ to $[0,\ell]^{d}+\ell s$.
\begin{lem}
\label{lem: phase observable}There are $C,c>0$ such that for each
$\ell,d\in\N$ and $\lambda>0$ there exists a $[0,\ell]^{d}$-local
function $f:\{0,1\}^{\mathbb{Z}^{d}}\to\{-1,0,1\}$ with the following
properties: Let $L\in2\ell\N$. Suppose that $\boldsymbol{\sigma}$
is sampled from $\mu^{\mathbb{Z}_{L}^{d}}$.
\end{lem}

\begin{enumerate}
\item \label{enu: separator}Deterministically, for each $s,t$ nearest
neighbors in $\mathbb{\Z}_{L/\ell}^{d}$, $\tau_{s}f\cdot\tau_{t}f\ge0$.
\item \label{enu:expectation0}$\E[\boldsymbol{f}]=0$ when $\ell$ is odd.
\item \label{enu: contour bound}If $d\ge2$, $\ell=3$ and $\lambda>\frac{C\log d}{d}$,
then for each $A\subset\Z_{L/\ell}^{d}$,
\begin{align}
\P(\tau_{s}\boldsymbol{f}=0\,\,\forall s\in A) & \le\left(5d(1+\lambda)^{-c\frac{\lambda}{(1+\lambda)d}\ell^{d}}\right)^{|A|}.\label{eq:mult_bound-1}
\end{align}
\end{enumerate}
\begin{rem}
Item \ref{enu: contour bound} holds for any $\ell\ge1$, however
the constant $c$ deteriorates as $\ell\to\infty$. To reduce clutter
we restrict to $\ell=3$ as this is the only value that we need for
the proof (any fixed odd $\ell>1$ would do).%
\end{rem}

We prove the lemma in Subsection \ref{subsec:Proof of contour bound lemma}.
Let us now explain how it implies Theorem \ref{thm: two Gibbs measures}.%

\subsection{\label{subsec:peierl}Proof of Theorem \ref{thm: two Gibbs measures}}
\begin{proof}
Let the dimension $d\ge2$. Fix $\ell=3$ and suppose that $\lambda>C_{0}\frac{\log d}{d}$
for $C_{0}$ sufficiently large for the following arguments. Let $f$
be the function from Lemma \ref{lem: phase observable}.

\textbf{Step 1}: Note that by taking $C_{0}$ large enough, the condition
on $\lambda$ for Part \ref{enu: contour bound} of Lemma \ref{lem: phase observable}
holds. By our choice of $\ell$ and and assumption on $\lambda$,
we may choose $C_{0}$ large enough so that the base of the exponent
on RHS of (\ref{eq:mult_bound-1}) is at most $c_{0}d^{-2}$ where
$c_{0}\to0$ as $C_{0}\to\infty$. It follows that for each $L\in2\ell\N$
and each $s\in\Z_{L/\ell}^{d}$,
\begin{align}
\max\{\mu^{\mathbb{Z}_{L}^{d}}(\boldsymbol{f}=0),\mu^{\mathbb{Z}_{L}^{d}}(\boldsymbol{f}\neq\tau_{s}\boldsymbol{f})\} & \le Cc_{0}d^{-2}\label{eq: Peierls bound}
\end{align}

For the event $\{\boldsymbol{f}=0\}$ this holds by (\ref{eq:mult_bound-1}).
For $\{\boldsymbol{f}\neq\tau_{s}\boldsymbol{f}\}$, this is a consequence
of the standard Peierls argument, as we briefly detail. To avoid discussing
non-contractible cycles, it is convenient to consider the cube subgraph
$P_{L/\ell-1}^{d}$ of $\Z_{L/\ell}^{d}$. Here, $P_{r}$ is the path
graph with vertex set $\{0,\dots,r\}$ connected in order. We say
that a set of vertices $A$ in $P_{L/\ell-1}^{d}$ is a minimal separating
set if it separates $0$ and $s$ in $P_{L/\ell-1}^{d}$ and is minimal
with respect to inclusion with this property. Let $\mathcal{A}_{n}$
be the family of minimal separating sets $A$ with $|A|=n$. The Peierls
argument relies on the fact that
\begin{equation}
|\mathcal{A}_{n}|\le(Cd^{2})^{n}.\label{eq: number of separators}
\end{equation}

This fact implies (\ref{eq: Peierls bound}) as follows: By Part \ref{enu: separator}
of Lemma \ref{lem: phase observable}, if $\tau_{s}f\neq f$ then
there exists $A\in\cup_{n}\mathcal{A}_{n}$ such that $\tau_{s}f=0$
for all $s\in A$. The union bound thus implies that $\mu^{\mathbb{Z}_{L}^{d}}(\tau_{s}\boldsymbol{f}\neq\boldsymbol{f})\le\sum_{A\in\mathcal{A}_{n}}\P(\tau_{s}\boldsymbol{f}=0\,\,\forall s\in A)$
and (\ref{eq: Peierls bound}) follows from (\ref{eq: number of separators})
and Part \ref{enu: contour bound} of Lemma \ref{lem: phase observable},
taking $C_{0}$ sufficiently large.

The bound (\ref{eq: number of separators}) is a fairly standard fact
about the cube $P_{L/\ell-1}^{d}$. We briefly indicate the ingredients
that go into it: First, by \cite[Lemma 2]{timar2013boundary}, every
$A\in\cup_{n}\mathcal{A}_{n}$ is connected in the graph $(P_{L/\ell-1}^{d})^{+}$,
where $(P_{L/\ell-1}^{d})^{+}$ is the graph which contains $P_{L/\ell-1}^{d}$
and adds edges between opposing vertices of each cycle of length $4$
in $P_{L/\ell-1}^{d}$. The factor $d^{2}$ in (\ref{eq: number of separators})
stems from the fact that the maximal degree in $(P_{L/\ell-1}^{d})^{+}$
is of order $d^{2}$. Second, if the graph distance between $A\in\cup_{n}\mathcal{A}_{n}$
and $\{0,s\}$ is large then $A$ itself has to be large (here we
use that $d\ge2$). For instance, it is simple to see that if the
distance is at least $r$ then $|A|\ge cr$ and this suffices for
our purposes.

\textbf{Step 2}: For each $L\in2\ell\N$, let $\mu_{L}$ be the measure
on $\Omega_{\hc}^{\mathbb{\Z}_{L}^{d}}$ obtained by conditioning
$\mu^{\mathbb{\Z}_{L}^{d}}$ on the event $\{\tau\boldsymbol{f}=1\}$,
for the element $\tau\in\refls^{d}$ which maps $[0,\ell]^{d}$ to
$[0,\ell]^{d}+\frac{1}{2}L(1,\ldots,1)$. We claim that
\begin{equation}
\inf_{L\in2\ell\N}\mu_{L}(\boldsymbol{f})>0.\label{eq: positive expectation}
\end{equation}
Indeed, 
\begin{align}
\mu_{L}(\boldsymbol{f}) & =\mu_{L}(\boldsymbol{f}=1)-\mu_{L}(\boldsymbol{f}=-1)\ge1-2\mu_{L}(\boldsymbol{f}\neq1)=1-2\mu_{L}(\boldsymbol{f}\neq\tau\boldsymbol{f})\nonumber \\
 & \ge1-2\frac{\mu^{\mathbb{Z}_{L}^{d}}(\boldsymbol{f}\neq\tau\boldsymbol{f})}{\mu^{\mathbb{Z}_{L}^{d}}(\tau\boldsymbol{f}=1)}=1-4\frac{\mu^{\mathbb{Z}_{L}^{d}}(\boldsymbol{f}\neq\tau\boldsymbol{f})}{1-\mu^{\mathbb{Z}_{L}^{d}}(\boldsymbol{f}=0)}\label{eq: expectation of f}
\end{align}

where the last equality uses that $\mu^{\mathbb{Z}_{L}^{d}}(\boldsymbol{f})=0$
(as $\ell$ is odd). Now, using (\ref{eq: Peierls bound}) with $c_{0}$
sufficiently small, we see that that (\ref{eq: expectation of f})
implies (\ref{eq: positive expectation}).%

The measure $\mu_{L}$ may be viewed as a measure on $\Omega_{\hc}^{\Z^{d}}$
supported on $L\Z^{d}$-periodic configurations. By compactness, the
sequence $(\mu_{L})_{L\in2\ell\N}$ has a weakly converging sub-sequence;
denote its limit by $\mu$. The fact that the box $[0,\ell]^{d}+\frac{1}{2}L(1,\ldots,1)$,
where the conditioning of $\mu_{L}$ is made, tends to infinity with
$L$ implies that $\mu$ is a Gibbs measure for the hard-core model
on $\Z^{d}$. Moreover, $\mu(\boldsymbol{f})>0$ due to (\ref{eq: positive expectation}).

However, we may also obtain a Gibbs measure $\mu'$ by a weak sub-sequential
limit of the measures $(\mu^{\mathbb{\Z}_{L}^{d}})_{L\in2\ell\N}$
(without conditioning). Since $\mu'(\boldsymbol{f})=0\neq\mu(\boldsymbol{f})$,
there are multiple Gibbs measures.
\end{proof}

\subsection{\label{subsec:Proof of contour bound lemma}Proof of Lemma \ref{lem: phase observable}}

Let $\ell\in\N$ and $L\in2\ell\N$. Let $\lambda>0$, $d\in\N$.

\subsubsection{Weighted sums}

Define a ``weight function'' $w:\Z^{d}\to\R$ supported on $[0,\ell]^{d}\cap\Z^{d}$,
by
\[
w_{v}\coloneqq\prod_{i=1}^{d}\begin{cases}
1/2 & v_{i}\in\{0,\ell\}\\
1 & v_{i}\in\{1,\dots\ell-1\}
\end{cases}.
\]
Note that, regardless of the value of $L$, it holds that $w_{v}^{-1}=|\Stab_{\refls[\ell][L]^{d}}(v)|$
where $\Stab_{\refls[\ell][L]^{d}}(v)=\{\tau\in\refls[\ell][L]^{d}:\tau v=v\}$
is the stabilizer of the element $v$ (when $v$ is viewed as an element
of $\Z_{L}^{d}$).

We use $w$ to define a weighted sum: for $\sigma\in\{0,1\}^{\mathbb{\Z}_{L}^{d}}$
and $A\subset\Z_{L}^{d}$, define
\[
|\sigma|_{A}^{w}\coloneqq\sum_{v\in A}w_{v}\sigma_{v}.
\]
The weight function is chosen so that whenever $A\subset\Z_{L}^{d}$
is invariant under the action of $\refls[l][L]^{d}$,
\begin{equation}
\sum_{\tau\in\refls[l][L]^{d}}|\sigma\tau|_{A}^{w}=|\sigma_{A}|.\label{eq:sums}
\end{equation}

\subsubsection{\label{subsec:def_of_nu}The local observable $f$}

The observable $f$ is defined as $f\coloneqq(1-\indic B)\cdot g$
for the ``bad'' event $B=\{f=0\}$ and observable $g$ that we now
define.

Let 
\[
\alpha\coloneqq c_{\alpha}\frac{\lambda}{1+\lambda}\ell^{d}
\]
for a small universal constant $c_{\alpha}$ to be determined later.
Define
\[
B_{\mathrm{0}}\coloneqq\left\{ \sigma:\min\{|\sigma|_{\Ve}^{w},|\sigma|_{\Vo}^{w}\}\ge\alpha\right\} .
\]
Define $\mathcal{F}$ to be the set of faces of $[0,\ell]^{d}$, discretized
and embedded in $\Z_{L}^{d}$ (each of the $2d$ elements of $\mathcal{F}$
is a set of $(\ell+1)^{d-1}$ vertices in $\Z_{L}^{d}$). For each
$H\in\mathcal{F}$ define
\[
B_{H}\coloneqq\left\{ \sigma:|\sigma|_{H}^{w}\le2\alpha\right\} .
\]
Define $B\coloneqq B_{0}\cup\bigcup_{H\in\mathcal{F}}B_{H}$ and define
$g$ as 
\[
g\coloneqq\sgn(|\sigma|_{\Vo}^{w}-|\sigma|_{\Ve}^{w})
\]
where $\sgn$ is the sign function.

Note that $f$,$g$, $\indic B$, $\indic{B_{0}}$, $\indic{B_{H}}$
are $[0,\ell]^{d}$-local functions that do not depend on $L$.

\subsubsection{The phases are separated by contours}
\begin{proof}
[Proof of item \ref{enu: separator} of Lemma \ref{lem: phase observable}]Let
$s,t$ be nearest neighbors in $\mathbb{\Z}_{L/\ell}^{d}$. Assume
for the sake of contradiction, and without loss of generality, that
$\tau_{s}f=1$, and $\tau_{t}f=-1$. By applying $\tau_{s}^{-1}$
to $\sigma$ we may assume without loss of generality that $f=1$,$\tau f=-1$
where $\tau=\tau_{s}^{-1}\tau_{t}$ is the reflection through some
face $H\in\mathcal{F}$. We have
\begin{align*}
\sigma\notin B_{H} & \iff|\sigma|_{H}^{w}>2\alpha,\\
g=1,\sigma\notin B_{0} & \implies|\sigma|_{\Ve}^{w}=\min\{|\sigma|_{\Ve}^{w},|\sigma|_{\Vo}^{w}\}<\alpha,\\
\tau g=-1,\sigma\tau\notin B_{0} & \implies|\sigma\tau|_{\Vo}^{w}=\min\{|\sigma\tau|_{\Ve}^{w},|\sigma\tau|_{\Vo}^{w}\}<\alpha.
\end{align*}
As vertices of $H$ are fixed by $\tau$, we obtain the contradiction
\[
2\alpha<|\sigma|_{H}^{w}=|\sigma|_{\Ve\cap H}^{w}+|\sigma|_{\Vo\cap H}^{w}=|\sigma|_{\Ve\cap H}^{w}+|\sigma\tau|_{\Vo\cap H}^{w}<\alpha+\alpha.\qedhere
\]
\end{proof}

\subsubsection{The anti-symmetry of $f$}
\begin{proof}
[Proof of item \ref{enu:expectation0} of Lemma \ref{lem: phase observable}]
Consider the reflection $\tau$ of $(\R/L\Z)^{d}$, defined by $v=(v_{1},\dots,v_{d})\mapsto(\ell-v_{1},v_{2},\dots,v_{d})$,
which maps $[0,\ell]^{d}$ to itself. While $\tau\notin\refls[\ell][L]^{d}$,
its action is still defined by (\ref{eq:actions}). Observe that $\tau\indic B=\indic B$.
However, when $\ell$ is odd, $\tau\Ve=\Vo$ and therefore $\tau g=-g$.
Let $\boldsymbol{\sigma}$ be sampled from $\mu^{\mathbb{Z}_{L}^{d}}$.
Finally, as $\boldsymbol{\sigma}$ and $\boldsymbol{\sigma}\tau$
have the same distribution, we conclude that $\E[\boldsymbol{f}]=\E[\tau\boldsymbol{f}]=\E[-\boldsymbol{f}]$,
so that $\E[\boldsymbol{f}]=0$.
\end{proof}

\subsubsection{Chessboard norm bounds}

Towards establishing item \ref{enu: contour bound} of Lemma \ref{lem: phase observable},
the next lemma uses Theorem \ref{thm:main_finite} (our main result
for finite graphs) to bound the chessboard norm of the ``bad'' event
$B$.

We proceed to prove item \ref{enu: contour bound} of Lemma \ref{lem: phase observable}.
For each $A\subset\Z_{L/\ell}^{d}$, 
\begin{align}
\P(\tau_{s}f=0\,\,\forall s\in A) & =\frac{\zeta^{\mathbb{\Z}_{L}^{d}}(\prod_{s\in A}\tau_{s}\indic B)}{Z^{\mathbb{\Z}_{L}^{d}}}\nonumber \\
\text{\ensuremath{\left(\substack{\text{by the chessboard estimate,}\\
\text{Proposition \ref{prop:chessboard}}
}
\right)} } & \le\frac{\zedd[L][\ell][\indic B]^{|A|}\zedd[L][\ell][1]^{\left(\frac{L}{\ell}\right)^{d}-|A|}}{Z^{\mathbb{\Z}_{L}^{d}}}\nonumber \\
 & \le\left(\zedd[L][\ell][\indic B]e^{-\frac{\ell^{d}}{2}\logplam}\right)^{|A|}\label{eq:mult_bound_computation}
\end{align}

where the in last inequality we used the trivial lower bound (\ref{eq:trivial partition function lower bound}),
\begin{align}
\zedd[L][\ell][1]^{\left(\frac{L}{\ell}\right)^{d}} & =Z^{\mathbb{\Z}_{L}^{d}}\ge e^{\frac{1}{2}\logplam L^{d}}=e^{\frac{\ell^{d}}{2}\logplam\left(\frac{L}{\ell}\right)^{d}}.\label{eq: Partition function lower bound-1}
\end{align}

Thus, item \ref{enu: contour bound} of Lemma \ref{lem: phase observable}
is implied by (\ref{eq:mult_bound_computation}) and the next lemma.
\begin{lem}
\label{lem:bad_norm}Suppose $d\ge2$ and $\ell=3$. There exist universal
$C,c,c_{\alpha}>0$ so that if $\lambda>\frac{C\log d}{d}$ then for
all $L\in2\ell\N$,
\begin{equation}
\zedd[L][\ell][\indic B]\le\zedd[2\ell][\ell][\indic B]\le5d\exp\left(\left(\frac{1}{2}-c\frac{\lambda}{(1+\lambda)d}\right)\ell^{d}\logplam\right).\label{eq:chessboard norm of bad event}
\end{equation}
\end{lem}

\begin{proof}
The first inequality follows from Lemma \ref{lem:chessboard_monotonicity}.
For the second inequality we consider the model on the graph $G=\mathbb{\Z}_{2\ell}^{d}$.
By the triangle inequality (Fact \ref{prop:seminorm}),
\begin{equation}
\zedd[2\ell][\ell][\indic B]\le\zedd[2\ell][\ell][\indic{B_{0}}]+\sum_{H\in\mathcal{F},\eps\in\{\pm1\}}\zedd[2\ell][\ell][\indic{B_{H,\eps}}]\label{eq: triangle inequality for B}
\end{equation}

where we define the events $B_{H,\eps}\coloneqq B_{H}\cap\{g=\eps\}\setminus B_{0}$.
We proceed to bound each of these terms.

\textbf{Bounding $\zedd[2\ell][\ell][\indic{B_{0}}]$}: Define the
event $\tilde{B}_{0}$ by $\indic{\tilde{B}_{0}}\coloneqq$$\prod_{\tau\in\refls[\ell][2\ell]^{d}}\tau\indic{B_{0}}$.
We first relate $\tilde{B}_{0}$ to $\min\{|\sigma_{\Ve}|,|\sigma_{\Vo}|\}$:
\begin{align*}
\sigma\in\tilde{B_{0}} & \iff\bigwedge_{\tau\in\refls[\ell][2\ell]^{d}}\left(\min\{|\sigma\tau|_{\Ve}^{w},|\sigma\tau|_{\Vo}^{w}\}>\alpha\right)\\
 & \implies\min\Bigg\{\sum_{\tau\in\refls[\ell][2\ell]^{d}}|\sigma\tau|_{\Ve}^{w},\sum_{\tau\in\refls[\ell][2\ell]^{d}}|\sigma\tau|_{\Vo}^{w}\Bigg\}>2^{d}\alpha\\
\text{\ensuremath{\left(\substack{\text{by (\ref{eq:sums}), as}\\
\text{\ensuremath{\Ve},\ensuremath{\Vo\ }are \ensuremath{\refls[][2\ell]^{d}}-invariant}
}
\right)}} & \iff\min\{|\sigma_{\Ve}|,|\sigma_{\Vo}|\}>2^{d}\alpha.
\end{align*}
We now apply Theorem \ref{thm:main_finite} with $G=\mathbb{\Z}_{2\ell}^{d}$
(which is $2d$-regular since $\ell=3$) and $r=2^{d}\alpha=c_{\alpha}\frac{\lambda}{1+\lambda}(2\ell)^{d}$.
We verify (\ref{eq:main_ass}) using that by Lemma \ref{lem: expansion for torus graphs},
$h(\Z_{2\ell}^{d})\ge c/\ell$ and the local expansion property holds
with $M_{\LE}=\frac{c(2\ell)^{d}}{d}$ and $C_{\LE}=C$. We also use
$\lambda>\frac{C\log d}{d}$ and let $C$ here depend on $c_{\alpha}$
which is chosen later.%

Thus, applying the theorem, we obtain for $\ell=3$ that
\begin{align}
\zedd[2\ell][\ell][\indic{B_{0}}] & =\sqrt[\left(\frac{2\ell}{\ell}\right)^{d}]{\zeta^{\mathbb{\Z}_{2\ell}^{d}}\left(\tilde{B_{0}}\right)}\nonumber \\
 & \le(1+\lambda)^{2^{-d}\left(\frac{|\mathbb{\Z}_{2\ell}^{d}|}{2}-c\frac{h(\Z_{2\ell}^{d})}{d}r\right)}\\
 & =\exp\left(\left(\frac{\ell^{d}}{2}-c\frac{h(\Z_{2\ell}^{d})}{d}\alpha\right)\logplam\right)\label{eq:B 0 bound}\\
 & \le\exp\left(\left(\frac{1}{2}-cc_{\alpha}\frac{\lambda}{(1+\lambda)d}\right)\ell^{d}\logplam\right).
\end{align}

\textbf{Bounding $\zedd[2\ell][\ell][\indic{B_{H,\eps}}]$}: By definition,
\begin{equation}
\zedd[2\ell][\ell][\indic{B_{H,\eps}}]=\sqrt[\left(\frac{2\ell}{\ell}\right)^{d}]{\zeta^{\mathbb{\Z}_{2\ell}^{d}}\left(\tilde{B}_{H,\eps}\right)}\label{eq:B H chessboard norm}
\end{equation}
 with 
\begin{align*}
\tilde{B}_{H,\eps} & \coloneqq\left\{ \sigma:\prod_{\tau\in\refls[\ell][2\ell]^{d}}\tau\indic{B_{H,\eps}}=1\right\} .
\end{align*}
Let $\boldsymbol{\sigma}$ be sampled from $\mu^{\mathbb{Z}_{2\ell}^{d}}$.
By (\ref{eq:I-zeta}), for any $K\subset\Z_{2\ell}^{d}$,
\begin{align}
\log\zeta^{\mathbb{\Z}_{2\ell}^{d}}\left(\tilde{B}_{H,\eps}\right) & =I(\boldsymbol{\sigma}\condon\tilde{B}_{H,\eps})\nonumber \\
\text{(by the subadditivity (\ref{eq: subadditivity of I}) of I)} & \le\sum_{v\in K}I(\boldsymbol{\sigma}_{v}\condon\tilde{B}_{H,\eps})+\sum_{v\in K^{c}}I(\boldsymbol{\sigma}_{v}\condon\tilde{B}_{H,\eps})\nonumber \\
\text{(concavity of \ensuremath{S} and Proposition \ref{prop: I<=00003Dlogplam})} & \le|K|\left(S(p_{K})+p_{K}\log\lambda\right)+|K^{c}|\tilde{\lambda}\label{eq: log event bound}
\end{align}

where
\[
p_{K}\coloneqq\frac{1}{|K|}\E\left[|\sigma_{K}|\condon\tilde{B}_{H,\eps}\right].
\]

To obtain a bound that will imply (\ref{eq:chessboard norm of bad event}),
we will exhibit a set $K$ whose size is suitably larger than $\frac{1}{2}|\Z_{2\ell}^{d}|$
on which the occupation fraction $p_{K}$ is negligible.

We choose $K$ according to $H\in\mathcal{F}$,$\eps\in\{\pm1\}$.
We describe this now for $H=\{0\}\times\{0,\dots,\ell\}^{d-1}$ and
$\eps=-1$; a similar argument applies in all other possibilities.
Set
\[
K\coloneqq\left(\{0\}\times\Z_{2\ell}^{d-1}\right)\cup\Vo=\bigcup_{\tau\in\refls[\ell][2\ell]^{d}}\tau(H\cup\Vo)
\]
and note that $|K|=\left(\frac{1}{2}+\frac{1}{4\ell}\right)(2\ell)^{d}$
and $|K^{c}|=\left(\frac{1}{2}-\frac{1}{4\ell}\right)(2\ell)^{d}$
(using that $d\ge2)$. Towards bounding $p_{K}$, by the definitions
of $g$, $B_{0}$ and $B_{H}$, 
\begin{align*}
\sigma\in\{g=\eps\}\setminus B_{0} & \implies|\sigma|_{\Vo}^{w}=\min\{|\sigma|_{\Ve}^{w},|\sigma|_{\Vo}^{w}\}<\alpha,\\
\sigma\in B_{H} & \iff|\sigma|_{H}^{w}\le2\alpha.
\end{align*}

Therefore,
\begin{align*}
\sigma\in\tilde{B}_{H,\eps} & \implies\bigwedge_{\tau\in\refls[\ell][2\ell]^{d}}\left(|\sigma\tau|_{\Vo}^{w}<\alpha,|\sigma\tau|_{H}^{w}\le2\alpha\right)\\
\text{(since \ensuremath{\Vo,H\subset K})} & \implies\bigwedge_{\tau\in\refls[\ell][2\ell]^{d}}|\sigma\tau|_{K}^{w}<3\alpha\\
 & \implies\sum_{\tau\in\refls[\ell][2\ell]^{d}}|\sigma\tau|_{K}^{w}<2^{d}\cdot3\alpha\\
\text{\ensuremath{\left(\substack{\text{by (\ref{eq:sums}), as}\\
 \text{\ensuremath{K} is\ensuremath{\refls[][2\ell]^{d}}-invariant} 
}
 \right)}} & \implies|\sigma_{K}|<2^{d}\cdot3\alpha
\end{align*}

which implies $p_{K}<2^{d}\cdot3\alpha/|K|\le6c_{\alpha}\frac{\lambda}{1+\lambda}\le\frac{\lambda}{1+\lambda}\le\min\{\logplam,1\}$
for small $c_{\alpha}$ (recalling $\tilde{\lambda}=\log(1+\lambda)$
from (\ref{eq: lambda tilde def})). Thus
\begin{align}
 & S(p_{K})+p_{K}\log\lambda\nonumber \\
\left(\substack{\text{by monotonicity for \ensuremath{p_{K}\le\frac{\lambda}{1+\lambda}}}\\
\text{and }(\ref{eq:binary_entropy_bound})
}
\right) & \le6c_{\alpha}\frac{\lambda}{1+\lambda}\log\frac{e\lambda}{6c_{\alpha}\frac{\lambda}{1+\lambda}}\nonumber \\
 & \le Cc_{\alpha}\frac{\lambda}{1+\lambda}(\logplam+\log\frac{C}{c_{\alpha}})\nonumber \\
\left(\frac{\lambda}{1+\lambda}\le\logplam\le1\right) & \le Cc_{\alpha}(1+\log\frac{C}{c_{\alpha}})\logplam\eqqcolon c_{\alpha}'\logplam\label{eq:I(K) bound}
\end{align}
and we note that $c_{\alpha}'\xrightarrow{c_{\alpha}\to0}0$. Finally,
continue (\ref{eq: log event bound}) to obtain
\begin{align}
\log\zeta^{\mathbb{\Z}_{2\ell}^{d}}\left(\tilde{B}_{H,\eps}\right) & \le|K|\left(S(p_{K})+p_{K}\log\lambda\right)+|K^{c}|\tilde{\lambda}\nonumber \\
\text{(by (\ref{eq:I(K) bound}))} & \le c_{\alpha}'|K|\logplam+|K^{c}|\logplam\nonumber \\
\text{(substitute \ensuremath{|K|,|K^{c}|})} & \le\left(c_{\alpha}'\left(\frac{1}{2}+\frac{1}{4\ell}\right)+\left(\frac{1}{2}-\frac{1}{4\ell}\right)\right)(2\ell)^{d}\logplam\le\nonumber \\
\text{(for small \ensuremath{c_{\alpha}}, and \ensuremath{\ell=3})} & \le\left(\frac{1}{2}-\frac{1}{15}\right)(2\ell)^{d}\logplam.\label{eq:tilde B bound}
\end{align}
\textbf{End of proof}: Combining (\ref{eq: triangle inequality for B}),
(\ref{eq:B 0 bound}), (\ref{eq:B H chessboard norm}) and (\ref{eq:tilde B bound}),
and fixing $c_{\alpha}$ yields:
\begin{align*}
\zedd[2\ell][\ell][\indic B] & \le\exp\left(\left(\frac{1}{2}-cc_{\alpha}\frac{\lambda}{(1+\lambda)d}\right)\ell^{d}\logplam\right)+4d\exp\left(\left(\frac{1}{2}-\frac{1}{15}\right)\ell^{d}\logplam\right)\\
 & \le5d\exp\left(\left(\frac{1}{2}-c\frac{\lambda}{(1+\lambda)d}\right)\ell^{d}\logplam\right).\qedhere
\end{align*}
\end{proof}

\part{Discussion and open questions\label{part:Discussion-and-open}}

\addtocounter{section}{1}
\setcounter{subsection}{0}

\subsection{Connections with earlier works}

Our Part \ref{part:general_finite} takes inspiration from \cite[Theorem 1.9]{kahnEntropyApproachHardCore2001}
\cite[Proposition 1.5]{galvin2004weighted}, where it is proved that
for every finite simple $\delta$-regular bipartite graph $G=(V,E)$,
it holds that
\begin{equation}
Z^{G}\le(Z^{K_{d,d}})^{\frac{|V|}{2d}}.\label{eq: K d d bound}
\end{equation}
Indeed, some of our proof steps can be seen as generalizations of
the proof steps there: Specifically, our definition of $\ph$ is the
same as the $Q$ defined in \cite{kahnEntropyApproachHardCore2001}.
Equations (3.1) and (3.3) of \cite{kahnEntropyApproachHardCore2001}
can then be viewed as a special case of (\ref{eq:shearer-1-1}) by
setting $\lambda=1$, $\boldsymbol{A}=\{\boldsymbol{v}\}$ and $\boldsymbol{B}=N(\boldsymbol{v})$,
with $\boldsymbol{v}$ a uniformly random element of $\Ve$. Similarly
(3.2) of \cite{kahnEntropyApproachHardCore2001} can be seen as the
first inequality of (\ref{eq:abst_even-1-1}).

The fact that the entropy proof of (\ref{eq: K d d bound}) in \cite{galvin2004weighted}
can be adapted to yield stronger bounds when applied to suitably restricted
subsets of configurations was a key point in \cite{peled2023rigidity}
(Lemma 4.7 there) and \cite{peled2020long} (Lemma 5.7 there). This
theme is present (in a very different form) in our proof of (\ref{eq:I<=00003DPhi bound}).

The use of the free energy functional $I$ in our work is inspired
by the (equivalent) use of Kullback--Leibler divergence in \cite{kozma2023lower}.

The use of chessboard estimates in Peierls-type arguments has been
common since its introduction to the field (see \cite{frohlichPhaseTransitionsReflection1980}).
More rare, however, are uses with reflection blocks of mesoscopic
size (so that control of the probability of disseminated bad events
becomes a significant task in its own right), and in this we took
inspiration from our earlier \cite{hadasColumnarOrderRandom2022}.

\subsection{\label{subsec: threshold for general graphs example}The threshold
fugacity for long-range order on general graphs}

Let $\delta\ge2$ integer and let $G=(V,E)$ denote a finite simple
$\delta$-regular bipartite graph, with bipartition $(\Ve,\Vo)$.
Let $\sigma$ be sampled from the hard-core measure $\mu^{G}$ at
fugacity $\lambda$. For which $\lambda$ does $\sigma$ exhibit long-range
order? As there is no canonical definition of long-range order for
such a finite graph, we may consider, as a proxy, the values of $\lambda$
for which the probability $\mu^{G}(\min\{|\sigma_{\Ve}|,|\sigma_{\Vo}|\}>r)$
is small when $r$ is of the order of $\frac{\lambda}{1+\lambda}|V|$.
For concreteness, let us define the set of \emph{balanced configurations}
as
\begin{equation}
E_{\text{Bal}}:=\{\sigma\in\Omega_{\hc}^{G}\colon\min\{|\sigma_{\Ve}|,|\sigma_{\Vo}|\}>\frac{1}{10}\frac{\lambda}{1+\lambda}|V|\}\label{eq: E bal def}
\end{equation}
and consider for which $\lambda$ does it hold that 
\begin{equation}
\mu^{G}(E_{\text{Bal}})<\frac{1}{10}.\tag{LRO}\label{eq: long-range order proxy}
\end{equation}

On the one hand, it is simple to see that (\ref{eq: long-range order proxy})
fails when $\lambda\le\frac{c_{0}}{\delta}$ with $c_{0}>0$ a sufficiently
small universal constant (as long as $\frac{\lambda}{1+\lambda}|V|\ge C$,
where, as usual, $C>0$ is a sufficiently large universal constant).

On the other hand, Theorem \ref{thm:main_finite} shows that (\ref{eq: long-range order proxy})
holds for all $\lambda>C\frac{\log\delta}{\delta}$ when $G$ satisfies
mild expansion properties (global expansion $h(G)\ge c$ and local
expansion with parameters $C_{\LE}\le C$ and $M_{\LE}\ge\frac{c\delta}{\log\delta}$
suffices, as long as $\frac{\lambda^{2}}{\delta}|V|\ge C$).

Thus, with mild assumptions on $G$, the threshold fugacity for long-range
order (in the sense of (\ref{eq: long-range order proxy})) lies between
order $\frac{1}{\delta}$ and order $\frac{\log\delta}{\delta}$.
Conjecture \ref{conj: threshold for Z^d} states that the threshold
fugacity has order $\frac{1}{\delta}$ for the (infinite) lattice
$\Z^{d}$, and we expect the same to be true for the (finite) hypercube
graph. It is tempting to think that the threshold fugacity should
always be of order $\frac{1}{\delta}$ under such assumptions. However,
we now show with an example that the threshold may have order $\frac{\log\delta}{\delta}$,
showing that the bound resulting from our main result, Theorem \ref{thm:main_finite},
cannot be improved without further assumptions.
\begin{example}
First, we define a ``linear gadget'' graph: The linear gadget with
a single block is \includegraphics[totalheight=1.5ex]{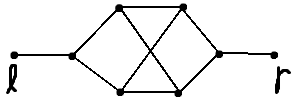} and
the linear gadget with two blocks is \includegraphics[totalheight=1.5ex]{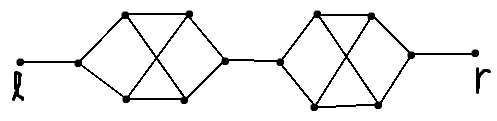}
. Similarly, for integer $m\ge1$, we let $L_{m}$ be the linear gadget
with $m$ blocks. The important features of $L_{m}$ are that it is
bipartite, the endpoint vertices $\ell,r$ in have degree $1$, all
other vertices have degree $3$ and its threshold fugacity for long-range
order is of order $m$ in the following sense: For $m\ge C$, at fugacity
at most $cm$, we have
\begin{equation}
\mu^{L_{m}}\left(\min\{|\sigma_{\Ve}|,|\sigma_{\Vo}|\}>\frac{1}{5}\frac{\lambda}{1+\lambda}|V(L_{m})\midvert\sigma_{\ell},\sigma_{r}\right)\ge\frac{99}{100}\label{eq: disorder on linear gadget}
\end{equation}
almost surely (on the occupancy status of the endpoint vertices $\ell,r$).

Second, we define a ``stretching operation'' by the linear gadget:
For integer $m\ge1$ and a $3$-regular graph $H$, let $\bar{H}^{m}$
be the graph formed from $H$ by replacing edge $\{u,v\}$ of $H$
with the linear gadget $L_{m}$, with $u,v$ identified with the endpoints
$\ell,r$ of $L_{m}$. Note that $\bar{H}^{m}$ is also $3$-regular,
and that it is bipartite when $H$ is. One may check that the Cheeger
constant satisfies $h(\bar{H}^{m})=\Theta(h(H)/m)$ (i.e., the ratio
of the two sides is between two positive universal constants).

Third, for integers $m\ge1$, $d\ge2$ and a $d$-regular graph $F$,
we let $F_{m}$ be the $m$-blow-up of $F$. That is, the graph with
$V(F_{m}):=\{(v,i)\colon v\in V(F),1\le i\le m\}$ and $(v,i)$ adjacent
to $(u,j)$ if $\{u,v\}\in E(F)$. Note that $F_{m}$ is a $dm$-regular
graph. One may check that the Cheeger constant satisfies $h(F_{m})=\Theta(mh(F))$,
and also that $F_{m}$ satisfies local expansion with parameters $C_{\LE}=\frac{1}{2}d^{2}$
and $M_{\LE}=|E(F)|$ (the latter fact follows by letting $(\boldsymbol{i}_{v})_{v\in F}$
uniform in $\{1,\ldots,m\}$ independently and taking the connected
graph $\boldsymbol{T}$ to be the induced graph on $\{(v,\boldsymbol{i}_{v})\colon v\in F\}$).

Finally, we define our example graph: Let $\delta$ be a large multiple
of $3$ and let $H$ be a $3$-regular bipartite expander graph (of
arbitrary size), i.e., having Cheeger constant $h(H)\ge c$ for a
universal $c>0$. Our example graph is $G:=(\bar{H}^{\delta})_{\delta/3}$
(the $\delta/3$-blow-up of the stretching of $H$ by the linear gadget
$L_{\delta}$). Observe that $G$ is a $\delta$-regular bipartite
graph with Cheeger constant $h(G)\ge c$ and local expansion with
parameters $C_{\LE}=\frac{9}{2}$ and $M_{\LE}=\frac{9}{2}\frac{|V(G)|}{\delta}$.
However, (\ref{eq: long-range order proxy}) fails for $G$ at fugacities
$\lambda\le\frac{c_{0}\log\delta}{\delta}$ for a sufficiently small
$c_{0}>0$. Indeed, the hard-core model on $G$ is equivalent to the
hard-core model on $\bar{H}^{\delta}$ (declaring a vertex vacant
if and only if all its copies in $G$ are vacant) with fugacity $\lambda'=(1+\lambda)^{\delta/3}-1\le\delta^{c_{0}/3}\ll\delta$,
whence the claim follows from the property (\ref{eq: disorder on linear gadget})
of the linear gadget.
\end{example}

\subsection{The probability of a balanced configuration on the hypercube}

We continue with the terminology of balanced configurations of the
previous section (see (\ref{eq: E bal def})). Theorem \ref{thm:main_finite}
bounds the probability of a balanced configuration for sufficiently
large fugacities. How sharp is this bound? We discuss this question
for the hypercube graph $G=\Z_{2}^{d}$.

We present two natural lower bounds on the probability of a balanced
configuration:
\begin{enumerate}
\item Identify the vertex set of the hypercube $\Z_{2}^{d}$ with $\{0,1\}^{d}$
and denote the Hamming weight of $v\in\Z_{2}^{d}$ by $|v|=\sum_{i=1}^{d}v_{i}$.
Let 
\[
E_{\text{Ham}}:=\{\sigma\in\Omega_{\hc}^{\Z_{2}^{d}}\colon\sigma_{v}=0\text{ for \ensuremath{v\in\Ve} with \ensuremath{|v|\le\frac{d}{2}} or \ensuremath{v\in\Vo} with \ensuremath{|v|>\frac{d}{2}}}\}
\]
be the set of hard-core configurations which have solely odd occupation
for vertices of Hamming weight at most $\frac{d}{2}$ and solely even
occupation for Hamming weight at vertices of Hamming weight more than
$\frac{d}{2}$. It is simple to check that
\[
\zeta^{\Z_{2}^{d}}(E_{\text{Ham}})\ge(1+\lambda)^{(\frac{1}{2}-\frac{C}{\sqrt{d}})2^{d}}.
\]
It is not hard to see that conditioning the hard-core model on $\Z_{2}^{d}$
to be in $E_{\text{Ham}}$ typically yields a balanced configuration
(say, with probability at least $\frac{1}{2}$) when $\lambda\gg\frac{1}{2^{d}}$.
Now, using Corollary $\ref{cor:Zd_SFE}$ we arrive at the following
lower bound
\begin{equation}
\mu^{\Z_{2}^{d}}\left(E_{\text{Bal}}\right)\ge\frac{1}{2}\frac{\zeta^{\Z_{2}^{d}}(E_{\text{Ham}})}{\zeta^{\Z_{2}^{d}}(\Omega_{\hc}^{\Z_{2}^{d}})}\ge c(1+\lambda)^{-\frac{C}{\sqrt{d}}2^{d}},\quad\lambda\ge C\frac{\log d}{d}.\label{eq: E bal first lower bound}
\end{equation}
\item A different lower bound is obtained by considering the set
\[
E_{\text{coord}}:=\{\sigma\in\Omega_{\hc}^{\Z_{2}^{d}}\colon\sigma_{v}=0\text{ for \ensuremath{v\in\Ve} with \ensuremath{v_{1}=0} or \ensuremath{v\in\Vo} with \ensuremath{v_{1}=1}}\}
\]
of hard-core configurations which have solely odd occupation at vertices
with first coordinate $0$ solely even occupation at vertices with
first coordinate $1$. Here we have
\begin{align*}
\zeta^{\Z_{2}^{d}}(E_{\text{\text{coord}}}) & =(1+2\lambda)^{\frac{1}{4}2^{d}}=(1+\lambda)^{\frac{1}{2}2^{d}}\left(1+\frac{\lambda^{2}}{1+2\lambda}\right)^{-\frac{1}{4}2^{d}}.
\end{align*}
Again, conditioning the hard-core model on $\Z_{2}^{d}$ to be in
$E_{\text{\text{coord}}}$ typically yields a balanced configuration
(say, with probability at least $\frac{1}{2}$) when $\lambda\gg\frac{1}{2^{d}}$.
Therefore, with Corollary $\ref{cor:Zd_SFE}$, we arrive at the lower
bound
\begin{equation}
\mu^{\Z_{2}^{d}}\left(E_{\text{Bal}}\right)\ge\frac{1}{2}\frac{\zeta^{\Z_{2}^{d}}(E_{\text{\text{coord}}})}{\zeta^{\Z_{2}^{d}}(\Omega_{\hc}^{\Z_{2}^{d}})}\ge c\left(1+\frac{\lambda^{2}}{1+2\lambda}\right)^{-\frac{1}{4}2^{d}},\quad\lambda\ge C\frac{\log d}{d}.\label{eq: E bal second lower bound}
\end{equation}
We note that this lower bound improves upon (\ref{eq: E bal first lower bound})
in the regime $C\frac{\log d}{d}\le\lambda\le\frac{c}{\sqrt{d}}$.
\end{enumerate}
We believe that these scenarios capture the leading behavior of the
(log of the) probability of the balanced event. Moreover, though we
have not shown the corresponding lower bound, this may hold also in
the wider range $\lambda\ge\frac{C}{d}$. Thus, we arrive at
\begin{conjecture}
There exist $C,d_{0}>0$ such that for all dimensions $d\ge d_{0}$,
\begin{equation}
\mu^{\Z_{2}^{d}}\left(E_{\text{Bal}}\right)=\begin{cases}
e^{-\Theta(\lambda^{2}2^{d})} & \frac{C}{d}\le\lambda\le\frac{1}{\sqrt{d}}\\
e^{-\Theta(\frac{\log(1+\lambda)}{\sqrt{d}}2^{d})} & \lambda\ge\frac{1}{\sqrt{d}}
\end{cases}\label{eq: E bal conj bound}
\end{equation}
with the $\Theta$ notation signifying that the bound is tight up
to universal multiplicative constants.
\end{conjecture}

Our results (Corollary \ref{cor:torus}) provide an upper bound on
$\mu^{\Z_{2}^{d}}\left(E_{\text{Bal}}\right)$ in the regime $\lambda\ge C\frac{\log d}{d}$,
which captures the behavior up to a power of $d$ in the exponent.

\subsection{Periodic Gibbs measures and simultaneous percolation}

As briefly mentioned in Section \ref{subsec:The-hard-core-model},
the hard-core model on $\Z^{d}$ admits infinite-volume limits with
even/odd boundary conditions, yielding two extremal Gibbs measures
$\mu^{\text{even}}$,$\mu^{\text{odd}}$, which are also invariant
to the parity-preserving automorphisms of $\Z^{d}$. Moreover, multiple
Gibbs measures exist if and only if these two measures are distinct,
so that our results show that $\mu^{\text{even}}\neq\mu^{\text{odd}}$
when $\lambda>C\frac{\log d}{d}$.

It is natural to ask whether, for a given fugacity $\lambda>0$, every
periodic (i.e., invariant to a full rank lattice of translations)
Gibbs measure is a mixture of $\mu^{\text{even}}$ and $\mu^{\text{odd}}$.
It may well be the case that this property holds for all $\lambda>0$
(an analogous property is known to hold at all temperatures in the
Ising model; see \cite{raoufi2020translation} and references therein).
However, our results do not directly imply this, even in the range
$\lambda>C\frac{\log d}{d}$, and the current state-of-the art is
that the property is known only when $\lambda>C\frac{\log^{3/2}d}{d^{1/4}}$
\cite[Section 3.1.1]{peled2020long} (or when the Gibbs measure is
unique).

Related to this is the notion of \emph{simultaneous percolation}:
When sampling a configuration $\sigma$ from a Gibbs measure $\mu$,
does an infinite connected component of odd-occupied vertices and
an infinite connected component of even-occupied vertices exist simultaneously,
where connectivity is measured at graph distance $2$? On the one
hand, almost sure absence of simultaneous percolation for a periodic
$\mu$ easily implies that $\mu$ is a mixture of $\mu^{\text{even}}$
or $\mu^{\text{odd}}$. On the other hand, occurrence of simultaneous
percolation is an obstacle to direct uses of the Peierls argument
or cluster expansion techniques (see also \cite[below Theorem 6]{jenssenIndependentSetsHypercube2020}).
In high dimensions, $\sigma|_{\Ve}$ (or $\sigma|_{\Vo})$ stochastically
dominates Bernoulli percolation (on $\Ve$ with connectivity at distance
2) of parameter $\frac{\lambda}{\lambda+(1+\lambda)^{2d}}$, whence
almost sure simultaneous percolation occurs when $\frac{C}{d^{2}}<\lambda<\frac{c\log d}{d}$
(as for \emph{minority percolation} in the Ising model \cite{aizenman1987percolation}).
We expect that there is no simultaneous percolation when $\lambda>\frac{C\log d}{d}$.

\section{Acknowledgments}

We are grateful to Wojtek Samotij for initial discussions of the problem
and the approach. We thank Gady Kozma, Wojtek Samotij and Yinon Spinka
for earlier collaborations on this and related problems, and Michal
Bassan, Quentin Dubroff, Jeff Kahn, Jonatan Hadas, Marcus Michelen,
Jinyoung Park and Dana Randall for several useful discussions.

Support from the National Science Foundation grant DMS-2451133, European
Research Council Consolidator grant 101002733 (Transitions) and Israel
Science Foundation grants 2110/22 and 2340/23 is gratefully acknowledged.

\bibliographystyle{alpha}
\bibliography{hard_core_bib}

\appendix

\section{Local expansion via random walks\label{sec:Local-expansion-via}}

In this section we prove Lemma \ref{lem: local expansion from Green function},
which bounds the local expansion parameters of a graph using properties
of simple random walk on that graph.

Let $G=(V,E)$ be a finite $\delta$-regular bipartite graph. Let
$C_{0}\ge1$ and let $M_{0}\ge1$ integer. Introduce the (finite length)
Green's function $g:V\times V\to[0,\infty)$:
\[
g_{u,w}=\E\left[|\{0\le i<M_{0}:\boldsymbol{W}_{i}=w\}|\right]
\]
with $\boldsymbol{W}_{0},\dots,\boldsymbol{W}_{M_{0}-1}$ a simple
random walk on $G$ starting at $\boldsymbol{W}_{0}=u$. Lemma \ref{lem: local expansion from Green function}
supposes that for each $v\in V$,
\begin{equation}
g_{v,v}-1\le\frac{C_{0}-1}{\delta}.\label{eq: diagonal Green function bound}
\end{equation}

We use the following simple consequence of the symmetry of the transition
matrix:
\begin{lem}
\label{lem:green_positivity}Let $u\neq w$ be vertices in the same
bipartite class of $G$. Then
\[
g_{u,w}\le\sqrt{(g_{u,u}-1)(g_{w,w}-1)}.
\]
\end{lem}

\begin{proof}
Let $P$ be the transition matrix for the simple random on $G$. Denote
$M\coloneqq\sum_{0<n<M_{0}/2}P^{2n}$ and note that $M$ is positive
semi-definite since $P$ is symmetric. Let $u\neq w$ be in the same
bipartite class of $G$ and observe that $M_{u,w}=M_{w,u}=g_{u,w}$,
whereas$M_{x,x}=g_{x,x}-1$ for all $x\in V$. Therefore, 
\[
0\le\det\left(\begin{matrix}M_{u,u} & M_{u,w}\\
M_{w,u} & M_{w,w}
\end{matrix}\right)=(g_{u,u}-1)(g_{w,w,}-1)-g_{u,w}^{2}.\qedhere
\]
\end{proof}
\begin{lem}
\label{lem:walk_expansion}Let $\boldsymbol{W}_{0},\dots,\boldsymbol{W}_{M_{0}-1}$
be a simple random walk on $G$ started at a uniformly random vertex
$\boldsymbol{W}_{0}\in V$. Then $\boldsymbol{W}_{i}$ is distributed
uniformly in $V$ for each $i$, and
\begin{equation}
\min_{v\in V}\text{\ensuremath{\P(\boldsymbol{W}\cap N(v)\neq\emptyset)\ge}\ensuremath{\frac{M_{0}}{C_{0}}\frac{\delta}{|V|}}}.\label{eq:path_vert}
\end{equation}
\end{lem}

\begin{proof}
The uniformity of $\boldsymbol{W}_{i}$ follows from the fact that
$G$ is $\delta$-regular. Let $v\in V$. Denote by $\boldsymbol{n}_{v}$
the number of times the random walk $\boldsymbol{W}$ visits $N(v)$,
i.e. $\boldsymbol{n}_{v}\coloneqq|\{i:\boldsymbol{W}_{i}\in N(v)\}|$.
Then, by the uniformity, $\E[\boldsymbol{n}_{v}]=\frac{M_{0}\delta}{|V|}$.
By the strong Markov property,
\begin{align*}
\E[\boldsymbol{n}_{v}\condon\boldsymbol{n}_{v}>0] & \le\sup_{u\in N(v)}\E[\boldsymbol{n}_{v}\condon\boldsymbol{W}_{0}=u]\\
 & =\sup_{u\in N(v)}\sum_{w\in N(v)}g_{u,w}\\
(\text{by Lemma \ref{lem:green_positivity}}) & \le\sup_{u\in N(v)}\Big[g_{u,u}+\sqrt{g_{u,u}-1}\sum_{w\in N(v)\setminus\{u\}}\sqrt{g_{w,w}-1}\Big]\\
(\text{by (\ref{eq: diagonal Green function bound})}) & \le C_{0}.
\end{align*}
Therefore
\[
\P(\boldsymbol{W}\cap N(v)\neq\emptyset)=\P(\boldsymbol{n}_{v}>0)=\frac{\E[\boldsymbol{n}_{v}]}{\E[\boldsymbol{n}_{v}\condon\boldsymbol{n}_{v}>0]}\ge\frac{M_{0}}{C_{0}}\frac{\delta}{|V|}.\qedhere
\]

Let $\boldsymbol{W}_{0},\dots,\boldsymbol{W}_{M_{0}-1}$ be as in
Lemma \ref{lem:walk_expansion}. To finish the proof of Lemma \ref{lem: local expansion from Green function}
it remains to verify that the graph $\boldsymbol{T}$ defined as the
graph spanned by the edges of the walk $\boldsymbol{W}$ satisfies
the properties in the definition of local expansion (Definition \ref{def:loc_exp})
with parameters $C_{\LE}=C_{0}$ and $M_{\LE}=M_{0}$. The fact that
each $\boldsymbol{W}_{i}$ is distributed uniformly in $V$ implies,
by the union bound, that $\P(uv\in E(\boldsymbol{T}))\le\frac{M_{0}}{|E|}$
for each $uv\in E$, while the fact that $\P(N(u)\cap V(\boldsymbol{T})\neq\emptyset)\ge\frac{M_{0}}{C_{0}}\frac{\delta}{|V|}$
for each $u\in V$ is exactly (\ref{eq:path_vert}).
\end{proof}

\section{Proofs of generalized Shearer's inequality\label{sec:proofs_of_shearer}}

The generalized Shearer's inequality (Proposition \ref{thm:shearer_gen})
has appeared in several places, in formulations that slightly differ
from ours, see \cite[Lemma 1]{rao2010notes}, \cite[Theorem A.2]{csokaEntropyExpansion2020}
and \cite[Lemma 2.2]{lovett2015notes}%
.

We provide here two proofs for Proposition \ref{thm:shearer_gen}.
\begin{proof}
For each fixed $K\subset J$, denote $\digamma(K)\coloneqq S((\boldsymbol{X}_{j})_{j\in K})$.
Then $\digamma$ is a sub-modular set-function:
\[
\digamma(K_{1}\cap K_{2})+\digamma(K_{1}\cap K_{2})\le\digamma(K_{1})+\digamma(K_{2}).
\]
For $g:J\to[0,\infty)$ define the functional $\hat{\digamma}$ by
the Choquet integral:
\[
\hat{\digamma}(g)=\int_{0}^{\infty}\digamma(\{j\in J:g(j)\ge t\})dt.
\]
Then by \cite[Corollary 4.10]{lovasz2023submodular} if follows that
$\hat{\digamma}$ is a convex functional. By Jensen's inequality it
holds that
\[
pS(\boldsymbol{X})=\hat{\digamma}(p\indic J)\le\E\hat{\digamma}(\indic{\boldsymbol{K}})=S((\boldsymbol{X}_{j})_{j\in\boldsymbol{K}}\condon\boldsymbol{K}).\qedhere
\]
\end{proof}
\begin{proof}
Write $X_{<j}$ for $X_{\{k:k<j\}}=(X_{k})_{k<j}$. Then
\begin{align*}
 & \quad pS(\boldsymbol{X})\\
\text{(chain rule)} & =\sum_{j\in J}pS(\boldsymbol{X}_{j}\condon\boldsymbol{X}_{<j})\\
\text{(\ensuremath{\P(j\in\boldsymbol{K})\ge p})} & \le\sum_{j\in K}\P(j\in\boldsymbol{K})S(\boldsymbol{X}_{j}\condon\boldsymbol{X}_{<j})\\
\Big(\P(j\in\boldsymbol{K})=\sum_{K\subset J}\P(\boldsymbol{K}=K)\indic{j\in K}\Big) & =\sum_{K\subset J}\P(\boldsymbol{K}=K)\sum_{j\in J}\indic{j\in K}S(\boldsymbol{X}_{j}\condon\boldsymbol{X}_{<j})\\
\text{(conditioning on less)} & \le\sum_{K\subset J}\P(\boldsymbol{K}=K)\sum_{j\in J}\indic{j\in K}S(\boldsymbol{X}_{j}\condon\boldsymbol{X}_{\{k\in K:k<j\}})\\
\text{(chain rule)} & =\sum_{K\subset J}\P(\boldsymbol{K}=K)S(\boldsymbol{X}_{k})\\
\text{(definition)} & =\sum_{K\subset J}\sum_{X}\P(\boldsymbol{K}=K)\P(\boldsymbol{X}_{K}=X)\log\frac{1}{\P(\boldsymbol{X}_{K}=X)}\\
\text{(independence)} & =\sum_{K\subset J}\sum_{X}\P(\boldsymbol{K}=K,\boldsymbol{X}_{K}=X)\log\frac{1}{\P(\boldsymbol{X}_{K}=X\condon\boldsymbol{K}=K)}\\
\text{(definition)} & =S(\boldsymbol{X}_{\boldsymbol{K}}\condon\boldsymbol{K}).\qedhere
\end{align*}
\end{proof}

\end{document}